\documentclass[a4paper]{amsart}

\title{Fan realizations for some~$2$-associahedra}

\thanks{TM was supported by a French doctoral grant Gaspard Monge of the {\'E}cole Polytechnique.}
\author{Thibault Manneville}
\address[TM]{LIX, \'Ecole Polytechnique, Palaiseau}
\email{thibault.manneville@lix.polytechnique.fr}

\usepackage{chancery}
\usepackage{fancybox,color,graphicx}
\usepackage{amssymb,amsfonts,amsthm,amsmath}
\usepackage[utf8]{inputenc}
\usepackage[T1]{fontenc}
\usepackage[french,english]{babel}
\usepackage[bookmarksopen]{hyperref}
\usepackage{mathrsfs,lmodern,float,textcomp,enumerate,tikz,multicol,xargs}
\usepackage{caption,paralist}
\usepackage[abs]{overpic}

\usepackage{array,booktabs,ctable,multirow,colortbl}
\newcolumntype{K}[1]{>{\centering\arraybackslash$}m{#1}<{$}}

\graphicspath{{figures/}}

\theoremstyle{plain}
\newtheorem{theorem}{Theorem}
\newtheorem{lemma}[theorem]{Lemma}
\newtheorem{proposition}[theorem]{Proposition}
\newtheorem{corollary}[theorem]{Corollary}

\newtheorem{observation}{Observation}
\theoremstyle{definition}

\theoremstyle{remark}

\newtheorem{remark}[theorem]{Remark}

\newtheorem{question}{Question}

\hypersetup{colorlinks=true, citecolor=darkblue, linkcolor=darkblue}

\definecolor{darkblue}{rgb}{0,0,0.7} 
\newcommand{\darkblue}{\color{darkblue}} 
\newcommand{\defn}[1]{\emph{\darkblue #1}} 

\newcommand{\apriori}{\textit{a priori }}

\newcommand{\Rmbb}{\mathbb{R}}\newcommand{\Nmbb}{\mathbb{N}}

\newcommand{\Nmcal}{\mathcal{N}}
\newcommand{\Pmcal}{\mathcal{P}}\newcommand{\Fmcal}{\mathcal{F}}

\newcommand{\ssm}{\smallsetminus} 
\newcommand{\eqdef}{\mbox{\,\raisebox{0.2ex}{\scriptsize\ensuremath{\mathrm:}}\ensuremath{=}\,}} 
\newcommand{\transpose}{^{-1}}

\newcommand{\polygon}{\Pmcal}

\newcommand{\multiassociahedron}{\mathbf{\Delta}}
\newcommandx{\ray}[1][1={(i,j)}]{\mathbf{v}_{#1}} 
\newcommand{\rays}{R} 
\newcommand{\e}{\mathbf{e}} 
\newcommand{\f}{\mathbf{f}} 
\newcommandx{\fanRealization}[1][1={2,n}]{\mathbf{\Fmcal_{#1}}}
\newcommand{\fan}{\Fmcal}

\newcommandx{\random}[1][1={i,j}]{\varepsilon_{(#1)}} 
\newcommandx{\coefLeft}[1][1={i,j}]{\lambda_{(#1)}} 
\newcommandx{\coefRight}[1][1={i,j}]{\rho_{(#1)}} 
\newcommandx{\coefA}[1][1={i,j}]{\alpha_{(#1)}} 
\newcommandx{\coefB}[1][1={i,j}]{\beta_{(#1)}} 
\newcommandx{\coefC}[1][1={i,j}]{\varepsilon_{(#1)}} 

\newcommand{\vertices}{\mathcal{V}} 
\newcommandx{\complex}[1][1=C]{\mathcal{#1}}
\newcommand{\network}{\Nmcal}
\newcommand{\segment}{\mathcal{I}}

\newcommand{\cone}{\Rmbb_{\geq0}} 
\DeclareMathOperator{\lk}{lk} 
\DeclareMathOperator{\st}{st} 
\DeclareMathOperator{\del}{del} 
\DeclareMathOperator{\stell}{stell} 
\DeclareMathOperator{\ops}{ops} 
\DeclareMathOperator{\join}{\ast} 

\newcommandx{\face}[1][1=f]{#1} 
\newcommandx{\facet}[1][1=F]{#1} 
\newcommandx{\ridge}[1][1=R]{#1} 

\newcommand{\sq}[1]{\mathsf{#1}} 
\newcommandx{\permutation}[1][1=\pi]{#1} 
\newcommand{\symmetricGroup}{\mathfrak{S}} 
\newcommand{\subwordComplex}{\mathcal{S}} 
\newcommand{\word}{\sq{Q}} 
\newcommandx{\factor}[1][1=U]{\sq{#1}} 
\newcommand{\commuteTo}{\sim} 
\DeclareMathOperator{\id}{Id} 

\newcommand{\rotated}{\circlearrowleft} 

\begin{document}

\begin{abstract}
A~$k$-associahedron is a simplicial complex whose facets, called~$k$-triangulations, are the inclusion maximal sets of diagonals of a convex polygon where no~$k+1$ diagonals mutually cross. Such complexes are conjectured for about a decade to have realizations as convex polytopes, and therefore as complete simplicial fans. Apart from four one-parameter families including simplices, cyclic polytopes and classical associahedra, only two instances of multiassociahedra have been geometrically realized so far. This paper reports on conjectural realizations for all~$2$-associahedra, obtained by heuristic methods arising from natural geometric intuition on subword complexes. Experiments certify that we obtain fan realizations of~$2$-associahedra of an~$n$-gon for~$n\in\{10,11,12,13\}$, further ones being out of our computational reach.
\medskip

\noindent\textsc{keywords.} Multiassociahedra $\cdot$ multitriangulations $\cdot$ subword complexes~$\cdot$~fans

\medskip
\noindent
\textsc{MSC classes.} 52B11, 52B12, 52B40, 05E45.
\end{abstract}

\maketitle

\captionsetup{width=.97\textwidth}

\section{Introduction}
\label{sec:introduction}

For integers~$k$ and~$n$, we consider a convex polygon~$\polygon$ with~$n+2k+1$ vertices and call a~\defn{$k$-triangulation} (or~\defn{multitriangulation} for unspecified~$k$) of~$\polygon$ any inclusion maximal set of diagonals such that no~$k+1$ of them mutually cross (see Figure~\ref{fig:2triangulations9gon}). As any diagonal with at most~$k-1$ vertices of~$\polygon$ on one side belongs to any~$k$-triangulation, we only consider the other diagonals, called~\defn{$k$-relevant diagonals}, as part of a~$k$-triangulation. The~\defn{$k$-associahedron}~$\multiassociahedron_{k,n}$ (or~\defn{multiassociahedron} for unspecified~$k$) is then the simplicial complex whose facets are the~$k$-triangulations of~$\polygon$. It was introduced by V.~Capoyleas and J.~Pach in~\cite{CapoyleasPach} where multitriangulations were studied as geometric graphs, after which the complex itself was independently shown to be pure by T.~Nakamigawa~\cite{Nakamigawa}, and A.~W.~M.~Dress, J.~H.~Koolen and V.~Moulton~\cite{DressKoolenMoulton}. It was also proved to be a piecewise linear sphere of dimension~$(kn-1)$ in an unpublished paper of J.~Jonsson~\cite{Jonsson-generalizedTriangulations}. Many structural aspects of multitriangulations, in particular their decomposition into~\defn{stars}, were then studied by V.~Pilaud and F.~Santos~\cite{PilaudSantos-multitriangulationsComplexesStarPolygons} in order to approach several open problems. Among them V.~Pilaud and F.~Santos recall a question first asked by J.~Jonsson~\cite{Jonsson-generalizedTriangulationsDiagonalFreeSubsetsStackPolyominoes} about geometric realizations of multiassociahedra.

\begin{question}
Are multiassociahedra boundary complexes of some convex polytopes?
\end{question}

Some instances of multiassociahedra turn out to be classical in polytope theory and therefore give a positive answer to this question in the following cases (see~\cite{PilaudSantos-multitriangulationsComplexesStarPolygons} for details, and~\cite{Ziegler} for a general background on polytopes).
\begin{compactitem}
 \item For~$n=0$, the complex~$\multiassociahedron_{k,0}$ is reduced to a single point (the empty set).
 \item For~$n=1$, the complex~$\multiassociahedron_{k,1}$ is the boundary of a~$k$-simplex.
 \item For~$n=2$, the complex~$\multiassociahedron_{k,2}$ is the boundary complex of a cyclic polytope.
 \item For~$k=1$, the complex~$\multiassociahedron_{1,n}$ is the complex of usual triangulations called (classical) associahedron, which was initially realized as a convex polytope by M.~Haiman~\cite{Haiman} and C.~Lee~\cite{Lee}, followed by many other explicit realizations (\cite{Loday,HohlwegLange,CeballosSantosZiegler} to cite a few, see Figures~\ref{fig:associahedraCeballosSantosZiegler} and~\ref{fig:associahedraLodayHohlwegLange}). 
\end{compactitem}

\begin{figure}
\centerline{\includegraphics[width=\textwidth]{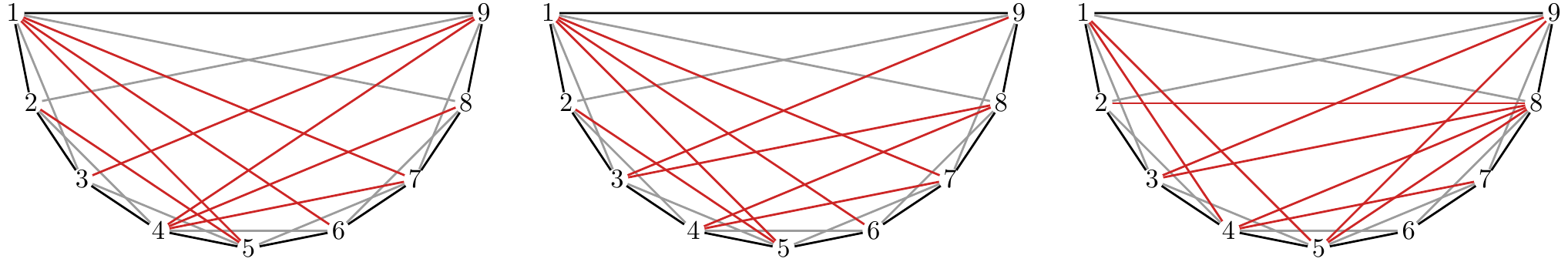}}
\caption{Three~$2$-triangulations of a~$9$-gon ($2$-relevant diagonals appear red).}
\label{fig:2triangulations9gon}
\end{figure}

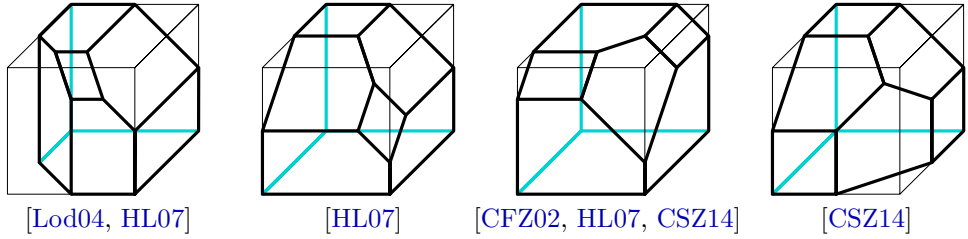
\begin{figure}
\centerline{\begin{overpic}[width=\textwidth]{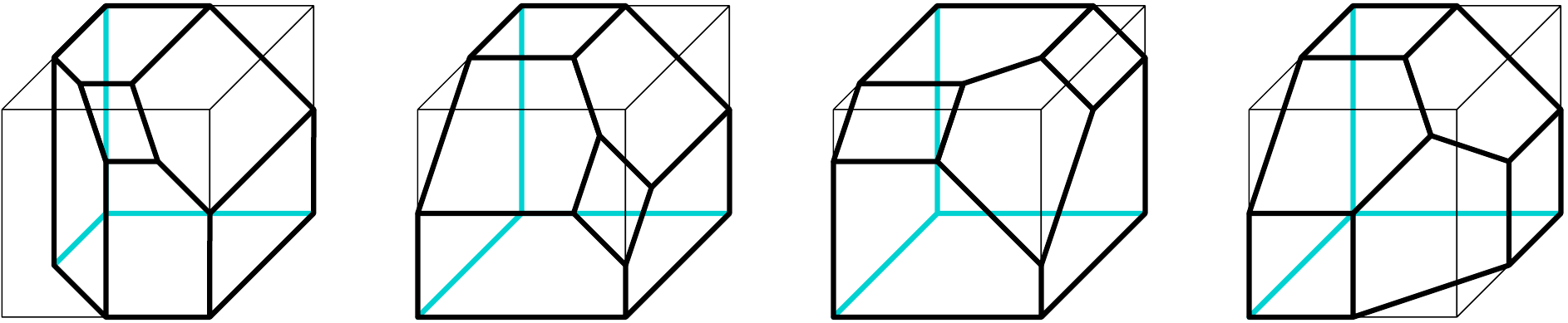}\put(7,-11){\cite{Loday,HohlwegLange}}\put(119,-11){\cite{HohlwegLange}}\put(175,-11){\cite{ChapotonFominZelevinsky,HohlwegLange,CeballosSantosZiegler}}\put(305,-11){\cite{CeballosSantosZiegler}}\end{overpic}}
\vspace{.4cm}
\caption{Several geometrically non equivalent polytopal realizations of the~$3$-dimensional classical associahedron. Figure from~\cite{CeballosSantosZiegler}, with permission.}
\label{fig:associahedraCeballosSantosZiegler}
\end{figure}

Apart from these classical realizations, a first oriented matroid theory approach allowed J.~Bokowski and V.~Pilaud to realize~$\multiassociahedron_{2,3}$ as a convex polytope in~\cite{BokowskiPilaud}. Using the framework of~\defn{sorting networks} as introduced by V.~Pilaud and M.~Pocchiola~\cite{PilaudPocchiola}, V.~Pilaud and F.~Santos then constructed brick polytopes in~\cite{PilaudSantos-brickPolytope} as an attempt to realize multiassociahedra. If these objects turned out to be interesting by themselves, none of them realized more multiassociahedra. C.~Stump observed in~\cite{Stump} the connection between multitriangulations and subword complexes as depicted by A.~Knutson and E.~Miller in~\cite{KnutsonMiller-subwordComplexesCoxeterGroups, KnutsonMiller-GrobnerGeometrySchubertPolynomials}.
C.~Ceballos, J.-P.~Labb{\'e} and C.~Stump then extended multiassociahedra to~\defn{multi cluster complexes} in any Coxeter type and developed combinatorial tools for subword complexes in~\cite{CeballosLabbeStump}. Finally N.~Bergeron, C.~Ceballos and J.-P.~Labb{\'e} used Gale duality in~\cite{BergeronCeballosLabbe} to realize as fans all complexes~$\multiassociahedron_{k,3}$ for~$k\in\Nmbb$, so as~$\multiassociahedron_{2,4}$ and~$\multiassociahedron_{3,4}$.

Using again subword complexes, we provide the first fan realizations of the complexes~$\multiassociahedron_{2,n}$ for~$n\in\{5,6,7,8\}$ and conjectural rays for any complex~$\multiassociahedron_{2,n}$ with~$n\in\Nmbb$ (see Question~\ref{que:fanRealizations}). All computations involved in this work were done with the software Sagemath~\cite{sage} (with available source code\footnote{\url{https://arxiv.org/abs/1608.08491}}).
We consider a convex~$(n+5)$-gon with vertices cyclically labeled from~$1$ to~$n+5$ and denote the diagonal between~$i\in[n+5]$ and~$j\ge i$ by~$(i,j)$. We denote by~$(\e_1,\dots,\e_n,\f_1,\dots,\f_n)$ the canonical basis of~$\Rmbb^{2n}$ and associate to each~$2$-relevant diagonal~$(i,j)$ a vector~$\ray$ in~$\Rmbb^{2n}$ as follows (see Theorem~\ref{thm:identificationMultiassociahedronSubwordComplex} and Figure~\ref{fig:identificationMultiassociahedronSubwordComplex}, and Figure~\ref{fig:2associahedronConjecturedPattern}).
\begin{compactenum}[\bf a)]
 \item $\ray[(1,4)] = \e_n - \f_n$ and~$\ray[(1,j+4)] = (2n + 2 - j) (\e_j - \e_{j+1}) + \e_n + \f_j - \f_n$  for~$j\in[n-1]$;
 \item $\ray[(2,j+4)] = \e_j + (2n + 2 - j) (\e_j - \e_{j+1}) + \f_j$ for~$j\in[n-1]$ and~$\ray[(2,n+4)] = \e_n + \f_n$;
 \item $\ray[(3,j+5)] = -\e_j$ for~$j\in[n]$;
 \item $\ray[(4,j + 6)] = \e_j + (2n + 2 - j) (\e_j - \e_{j+1}) - \f_j$ for~$j\in[n-1]$;
 \item $\ray[(i+4,i+j+6)] = j\,\e_i-(j-1)(\e_{i+j} + \e_{i+j+1}) + (2n+4-i)(\e_{i+j} - \e_{i+1}) + \f_i - \f_{i+j}$ for~$i\in[n-2]$ and~$j\in[n-i-1]$.
\end{compactenum}

\begin{theorem}
\label{thm:fanRealizations}
The vectors~$\ray$ are the rays of a complete simplicial fan in~$\Rmbb^{2n}$ which realizes the multiassociahedron~$\multiassociahedron_{2,n}$ for~$n\in[8]$.
\end{theorem}

\begin{question}
\label{que:fanRealizations}
Are the vectors~$\ray$ the rays of a complete simplicial fan in~$\Rmbb^{2n}$ which realizes the multiassociahedron~$\multiassociahedron_{2,n}$ for any~$n\ge1$?
\end{question}

\noindent We explain later why we state Question~\ref{que:fanRealizations} as a question rather than as a conjecture.

\begin{figure}
\centerline{\includegraphics[width=.82\textwidth]{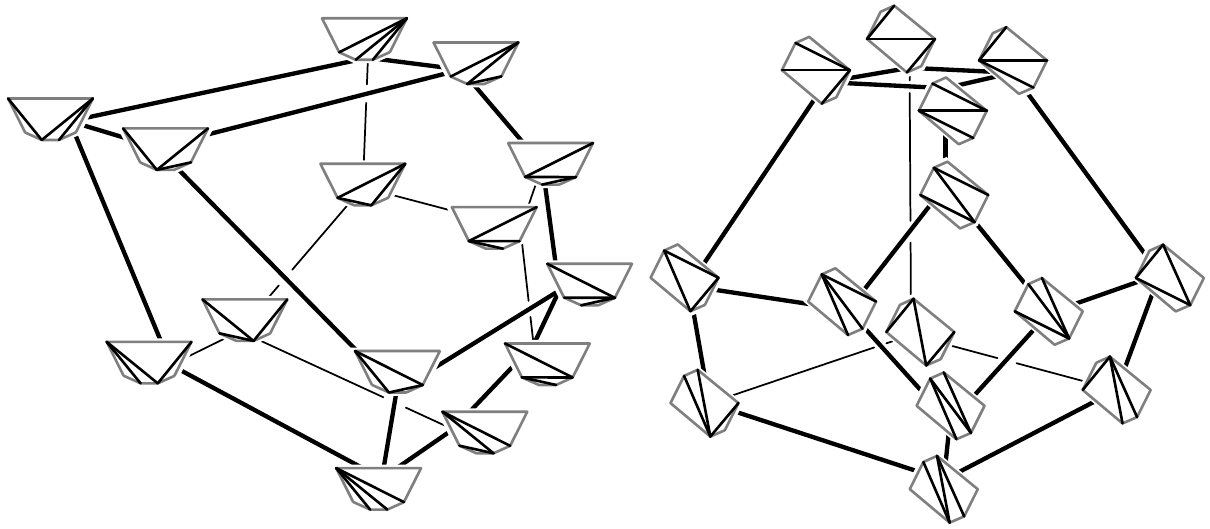}}
\vspace{-.2cm}
\caption{Two realizations of the~$3$-dimensional classical associahedron by C.~Hohlweg and C.~Lange~\cite{HohlwegLange} (see also Figure~\ref{fig:associahedraCeballosSantosZiegler} left). The left one is initially due to J.-L.~Loday~\cite{Loday}. Figure from~\cite{LangePilaud}, with permission.}
\vspace{-.3cm}
\label{fig:associahedraLodayHohlwegLange}
\end{figure}

Theorem~\ref{thm:fanRealizations} was checked computationally using the characterization of complete simplicial fans of Proposition~\ref{prop:characterizationSimplicialFans}. The rest of the paper will therefore mostly be a report on the heuristic process leading to the candidate rays. We describe in Section~\ref{sec:preliminaries} the notions and properties about simplicial complexes, polyhedral geometry and subword complexes that we need. In Section~\ref{sec:associahedra} we obtain by a new method the realization of the associahedron by J.-L.~Loday~\cite{Loday}. This method is the starting point of our heuristic construction of~$2$-associahedra as fans, which is presented in Section~\ref{sec:2associahedra}. Finally we briefly discuss some further aspects of our work in Section~\ref{sec:discussion}.

\section{Preliminaries}
\label{sec:preliminaries}

Our work relies on the interpretation of multiassociahedra as type~$A$ subword complexes, as stated in Theorem~\ref{thm:identificationMultiassociahedronSubwordComplex} (see Section~\ref{subsec:subwordComplexes}). Our main tools will be combinatorial operations on them studied by M.~Gorsky in~\cite{Gorsky-nilHeckeMoves,Gorsky-braidMoves}, and that we will try to translate geometrically. We present all the notions we need on simplicial complexes, polyhedral geometry and subword complexes in Sections~\ref{subsec:simplicialComplexes},~\ref{subsec:polyhedralGeometry} and~\ref{subsec:subwordComplexes} respectively. The reader familiar with them can proceed directly with Section~\ref{subsec:operationsSubwordComplexes}.

\subsection{Simplicial complexes}
\label{subsec:simplicialComplexes}

Given a finite set~$\vertices$, a~\defn{simplicial complex} (or a~\defn{complex})~$\complex$ on~$\vertices$ is a subset of the power set of~$\vertices$ closed under taking subsets:~${\complex\subseteq2^{\vertices}}$ and~$\face\subseteq\face[g]\in\complex\implies \face\in\complex$. Usually one requires~$\complex$ to contain all singletons. The elements of~$\vertices$, and by extension the corresponding singletons, are the~\defn{vertices} of~$\complex$. The pairs in~$\complex$ are the~\defn{edges} of~$\complex$ and form together with the vertices a graph called the~\defn{$1$-skeleton} of~$\complex$. The elements of~$\complex$ are its~\defn{faces}, the inclusion-maximal of which are called~\defn{facets}. We will always describe any explicit complex by its list of facets, which is equivalent to the whole data. We will moreover denote a complex whose single facet is an edge~$\{x,y\}$ directly by~$xy$, and we will use the notation~$x$ both for the vertex~$x$ and for the singleton~$\{x\}$. If~$\complex=2^{\vertices}$, then~$\complex$ is called a~\defn{simplex}. In particular any face of a simplicial complex is the unique facet of a simplex, therefore the faces of~$\complex$ are also called the simplices of~$\complex$. The~\defn{dimension} of a face~$\face\in\complex$ is the quantity~$\dim(f)\eqdef|\face|-1$ while the dimension of~$\complex$ is~$\dim(\complex)\eqdef\max_{\face\in\complex}\dim(\face)$. The complex~$\complex$ is~\defn{pure} if all its facets have the same cardinality~$d+1\ge1$, in which case~$\complex$ is also a~$d$-complex. The faces of dimension~$(d-1)$ of a~$d$-complex are called its~\defn{ridges}. Given a face~$\face$ of~$\complex$, the~\defn{star}~$\st_{\complex}(\face)$, the~\defn{link}~$\lk_{\complex}(\face)$ and the \defn{deletion}~$\del_{\complex}(\face)$ of~$\face$ in~$\complex$ are the complexes respectively defined by
\begin{align*}
\st_{\complex}(\face)  \eqdef & \{\face'\in\complex\,|\,\face\cup \face'\in\complex\},  \\
\lk_{\complex}(\face)  \eqdef & \{\face'\in\complex\,|\,\face\cap \face'=\varnothing\text{ and }\face\cup \face'\in\complex\},  \\
\del_{\complex}(\face)  \eqdef & \{\face'\in\complex\,|\,\face\not\subseteq \face'\}.
\end{align*}

Any simplicial complex~$\complex$ can be associated to a topological space called its~\defn{topological realization}, unique up to homeomorphism, obtained by gluing topological simplices along faces given by~$\complex$. The complex~$\complex$ is a~\defn{simplicial sphere} (or just a~\defn{sphere}) if it is pure of dimension~$d$ and its topological realization is homeomorphic to the standard sphere in~$\Rmbb^{d+1}$. Any ridge of a simplicial sphere is the intersection of exactly two facets of this simplicial sphere.

We now describe the classical operations on simplicial complexes that we need. Given two complexes~$\complex_1$ and~$\complex_2$, the~\defn{join} of~$\complex_1$ and~$\complex_2$ is the complex
\[
\complex_1\join\complex_2 \eqdef \{\face\sqcup\face'\,|\,\face\in\complex_1,\face'\in\complex_2\}
\]
where the complexes~$\complex_1$ and~$\complex_2$ are considered with disjoint sets of vertices and~$\sqcup$ denotes the disjoint union. The~\defn{suspension} of a complex~$\complex$ is the join of~$\complex$ with a complex consisting in two singletons, called~\defn{suspension vertices}. Given a vertex~$x$ of~$\complex$, the~\defn{one-point-suspension} of~$\complex$ with respect to~$x$ is the complex
\[
\ops_{\complex}(x) \eqdef \big(\del_{\complex}(x)\join\{x_0,x_1\}\big) \cup \big(\lk_{\complex}(x)\join x_0x_1\big)
\]
where~$x_0$ and~$x_1$ are two new vertices, also called~\defn{suspension vertices} in~$\ops_{\complex}(x)$. This operation extends the usual suspension: in the particular case where the vertex~$x$ is only contained in the face~$\{x\}$ of~$\complex$, then we consider by convention that the right part of the union is empty and the left part is just the suspension of the complex where the disconnected vertex~$x$ has been forgotten. So the suspension of a complex is obtained by adding an artificial disconnected vertex to it and taking the one-point-suspension with respect to this vertex. Figure~\ref{fig:onePointSuspension} illustrates the one-point-suspension operation on two complexes.
For a complex~$\complex$ and a face~$\face$ of~$\complex$, the~\defn{stellar subdivision} of the face~$\face$ in~$\complex$ is the complex
\[
\stell_{\complex}(\face)\eqdef \del_{\complex}(\face) \cup \{\face'\cup\{a\}\,|\,\face\not\subseteq\face'\in\st_{\complex}(f)\} = \del_{\complex \cup \st_{\complex}(\face)\join\{a\}}(\face)
\]
where~$a$ is a new vertex, called~\defn{subdivision vertex}. Intuitively the stellar subdivision corresponds to ``putting a vertex in the middle of the face~$\face$'' and adding the faces necessary to preserve the topology of the complex (see Figure~\ref{fig:stellarSubdivision} for examples).

\begin{figure}
\centerline{\includegraphics[width=1\textwidth]{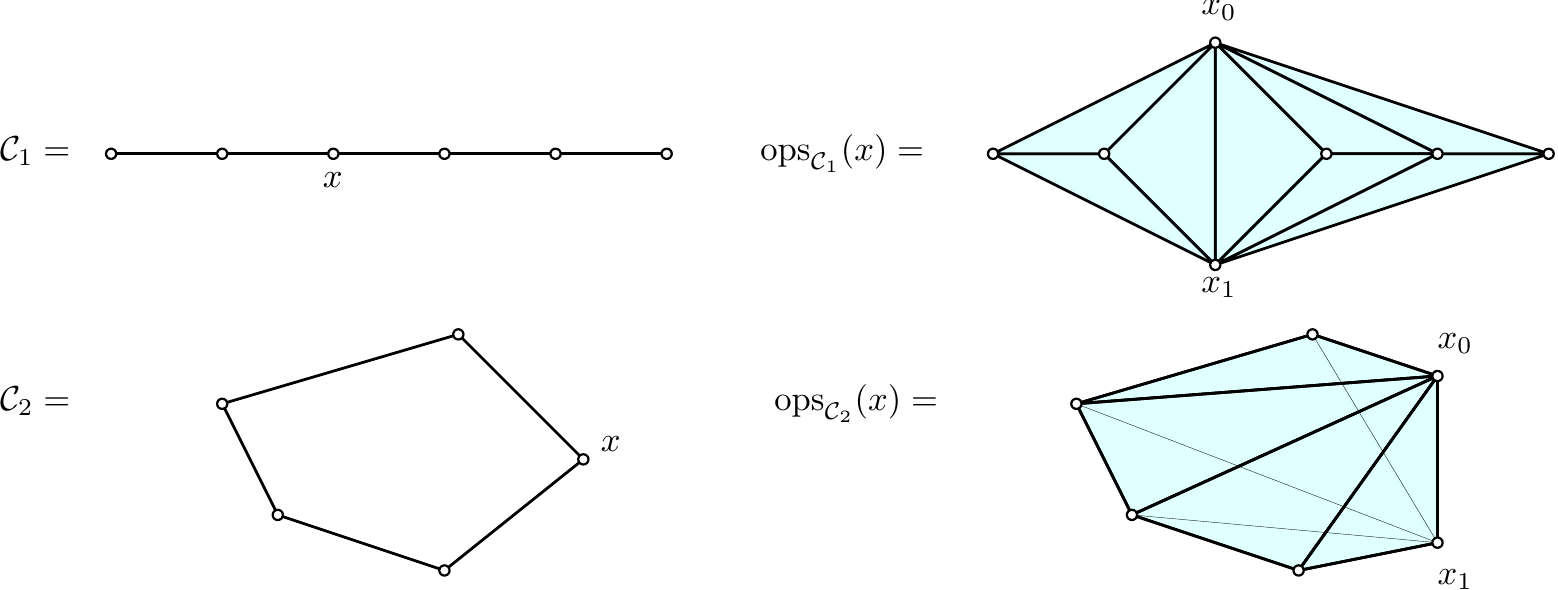}}
\caption{Two~$1$-dimensional complexes~$\complex_1$ and~$\complex_2$ (left) and their one-point-suspensions~$\ops_{\complex_1}(x)$ and~$\ops_{\complex_2}(x)$ with respect to a given vertex~$x$ (right).}
\label{fig:onePointSuspension}
\end{figure}

\begin{figure}
\centerline{\includegraphics[width=1\textwidth]{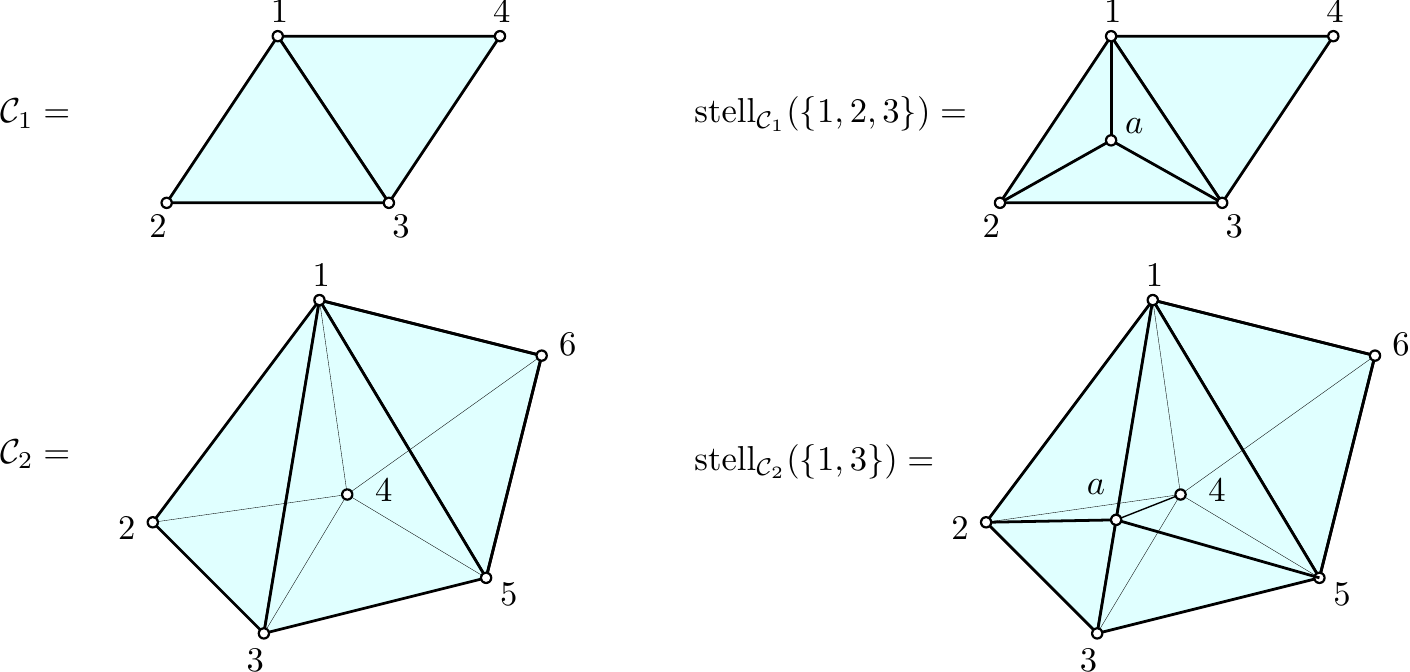}}
\caption{A~$2$-dimensional complex~$\complex_1=\{\{1,2,3\},\{1,3,4\}\}$ (top left) and the stellar subdivision of the facet~$\{1,2,3\}$ in it (top right), and a~$3$-dimensional complex~$\complex_2=\{\{1,2,3,4\},\{1,3,4,5\},\{1,4,5,6\}\}$ (bottom left) and the stellar subdivision of the edge~$\{1,3\}$ in it (bottom right).}
\label{fig:stellarSubdivision}
\end{figure}

\subsection{Polyhedral geometry}
\label{subsec:polyhedralGeometry}

We now briefly recall some notions of polyhedral geometry. We refer to the textbooks~\cite[Lecture 1]{Ziegler} and~\cite[Section 2.1.1.]{DeLoeraRambauSantos} for a complete presentation. Given a set~$\mathbf{V}$ of vectors in~$\Rmbb^n$, we will denote by~$\cone\mathbf{V}$ the positive span of~$\mathbf{V}$ in~$\Rmbb^n$. A~\defn{polyhedral cone} (or just a~\defn{cone}) is a subset of~$\Rmbb^n$ equivalently defined as the positive span of finitely many vectors or as the intersection of finitely many closed linear halfspaces. The~\defn{dimension} of a cone is the dimension of its linear span. The~\defn{faces} of a cone are its intersections with its~\defn{supporting hyperplanes}, that is the linear hyperplanes that do not strictly separate two of its elements. Faces of a cone still are cones and the~$1$-dimensional faces of a cone are its~$\defn{rays}$. A cone is~\defn{simplicial} if it is generated by an independent set of vectors. A simplicial cone is generated by its rays and any subset of rays generates a face.

A~\defn{(polyhedral) fan} is a set of cones closed by taking faces and such that any two of them intersect in a common face. The full dimensional faces of the fan are its~\defn{facets}. A fan is~\defn{simplicial} if all its cones are, and it is~\defn{complete} if the union of its cones covers the whole space~$\Rmbb^n$. A simplicial fan~$\fan$ can naturally be seen as an abstract simplicial complex~$\complex_{\fan}$ whose vertices are the rays of its cones and whose faces are the subsets of rays generating the cones of~$\fan$. An abstract simplicial complex~$\complex$ is then~\defn{realizable as a fan} (or a~\defn{geodesic sphere}) if there exists a complete simplicial fan~$\fan$ such that~$\complex$ is isomorphic to~$\complex_{\fan}$. The fan~$\fan$ is then called a~\defn{fan realization} of the complex~$\complex$. To realize a complex as a complete simplicial fan, we only need to find suitable coordinates for the rays corresponding to its vertices. These vectors then support a complete simplicial fan if a certain condition on~\defn{adjacent} facets (that is which differ by a single element of the complex) is satisfied.

\begin{proposition}[{see \textit{e.g.} \cite[Corollary 4.5.20.]{DeLoeraRambauSantos}}]
\label{prop:characterizationSimplicialFans}
Let~$\complex$ be a sphere with vertex set~$\vertices$ and~$\mathbf{V} \eqdef (\mathbf{v}_x)_{x \in \vertices}$ be a set of vectors in~$\Rmbb^n$. The set of cones~$\left\{\cone\mathbf{V}_{\face}\,|\,\face \in \complex\right\}$, where~$\mathbf{V}_{\face} \eqdef \left\{\mathbf{v}_x\,|\,x \in \face\right\}$, forms a complete simplicial fan if and~only~if
\begin{enumerate}
\item there exists a facet~$\facet$ of~$\complex$ such that~$\mathbf{V}_{\facet}$ is a basis of~$\Rmbb^n$ and such that the open cones~$\Rmbb_{> 0}\mathbf{V}_{\facet}$ and~$\Rmbb_{> 0}\mathbf{V}_{\facet'}$ are disjoint for any other facet~$\facet'$ of~$\complex$;
\item for any two adjacent facets~$\facet, \facet'$ of~$\complex$ with~$\facet \ssm \{x\} = \facet' \ssm \{x'\}$, the coefficients~$\alpha, \alpha'$ in the unique (up to rescaling) linear dependence
\[
\alpha \, \mathbf{v}_x + \alpha' \, \mathbf{v}_{x'} + \sum_{y \in \facet \cap \facet'} \beta_y \, \mathbf{v}_y = 0
\]
on~$\mathbf{V}_{\facet \cup \facet'}$ have the same sign (different from~$0$).
\end{enumerate}
\end{proposition}

Realizing a complex as a complete simplicial fan is in fact weaker than as a polytope, but we will skip details on that part since this paper does not deal with polytopal realizations. We conclude this section with some realizability results.

\begin{lemma}[folklore]
\label{lem:realizability}
One-point-suspensions and stellar subdivisions of simplicial complexes realizable as polytopes (resp. complete simplicial fans, resp. spheres) still are realizable as simplicial polytopes (resp. complete simplicial fans, resp. spheres).
\end{lemma}

Lemma~\ref{lem:realizability} is classical and its proof is left to the reader. We only describe the actual transformation on the rays of a complete simplicial fan allowing to realize as well the stellar subdivisions and the one-point-suspensions of the corresponding complex. We will indeed need them to derive the coordinates of Question~\ref{que:fanRealizations}. Let~$\fan$ be a complete simplicial fan in~$\Rmbb^n$ and let~$\rays$ be a set of vectors in~$\Rmbb^n$ such that for each ray~$\rho$ of~$\fan$ there is exactly one vector~$\ray[]\in\rays$ such that~$\rho=\cone\ray[]$. The vertex set of the simplicial complex~$\complex_\fan$ associated to the fan~$\fan$ can then naturally be identified with the set~$\rays$. Let~$\face=\{\ray[1],\dots,\ray[\ell]\}$ be a face of the complex~$\complex_\fan$. Then the complex~$\stell_{\complex_\fan}(\face)$ can be realized as a complete simplicial fan by adding a ray to the fan~$\fan$, generated by any vector of the form~$\alpha_1\ray[1]+\dots+\alpha_\ell\ray[\ell]$ with~$\alpha_1>0,\dots,\alpha_\ell>0$. This new ray corresponds to the subdivision vertex of the stellar subdivision. The generic choice consists to set all~$\alpha_i$'s equal to~$1$ (see Figure~\ref{fig:operationsOnFans}). Let~$\ray[]$ be a vector in~$\rays$, then the complex~$\ops_{\complex_\fan}(\ray[])$ is of dimension one more than the complex~$\complex_\fan$. We consider the vector space~$\Rmbb^{n+1}\eqdef\Rmbb^n\oplus\Rmbb\e_{n+1}$ and associate to a vector~$\ray[]'\in\rays\ssm\{\ray[]\}$ the vector~$\ray[]'\oplus\mathbf{0}$. The suspensions vertices obtained from~$\ray[]$ are associated to two vectors~$\ray[]\oplus\alpha\e_{n+1}$ and~$\ray[]\oplus\beta\e_{n+1}$, with~$\alpha\beta<0$. The generic choice for us will be~$\alpha=-1$ and~$\beta=1$. The set of rays that we obtain in~$\Rmbb^{n+1}$ then supports a complete simplicial fan realizing~$\ops_{\complex_\fan}(\ray[])$ (see Figure~\ref{fig:operationsOnFans}). In the particular case of a suspension, one can artificially add the zero vector~$\mathbf{0}$ to the set~$\rays$ and choose~$\ray[]=\mathbf{0}$  in the previous construction (see Figure~\ref{fig:operationsOnFans}). The previous descriptions give valid coordinates but certainly not all of them. Yet these realizations are easy to implement and they will be enough for our purposes.

\begin{figure}
\centerline{\hspace{2cm}\begin{overpic}[width=1.2\textwidth]{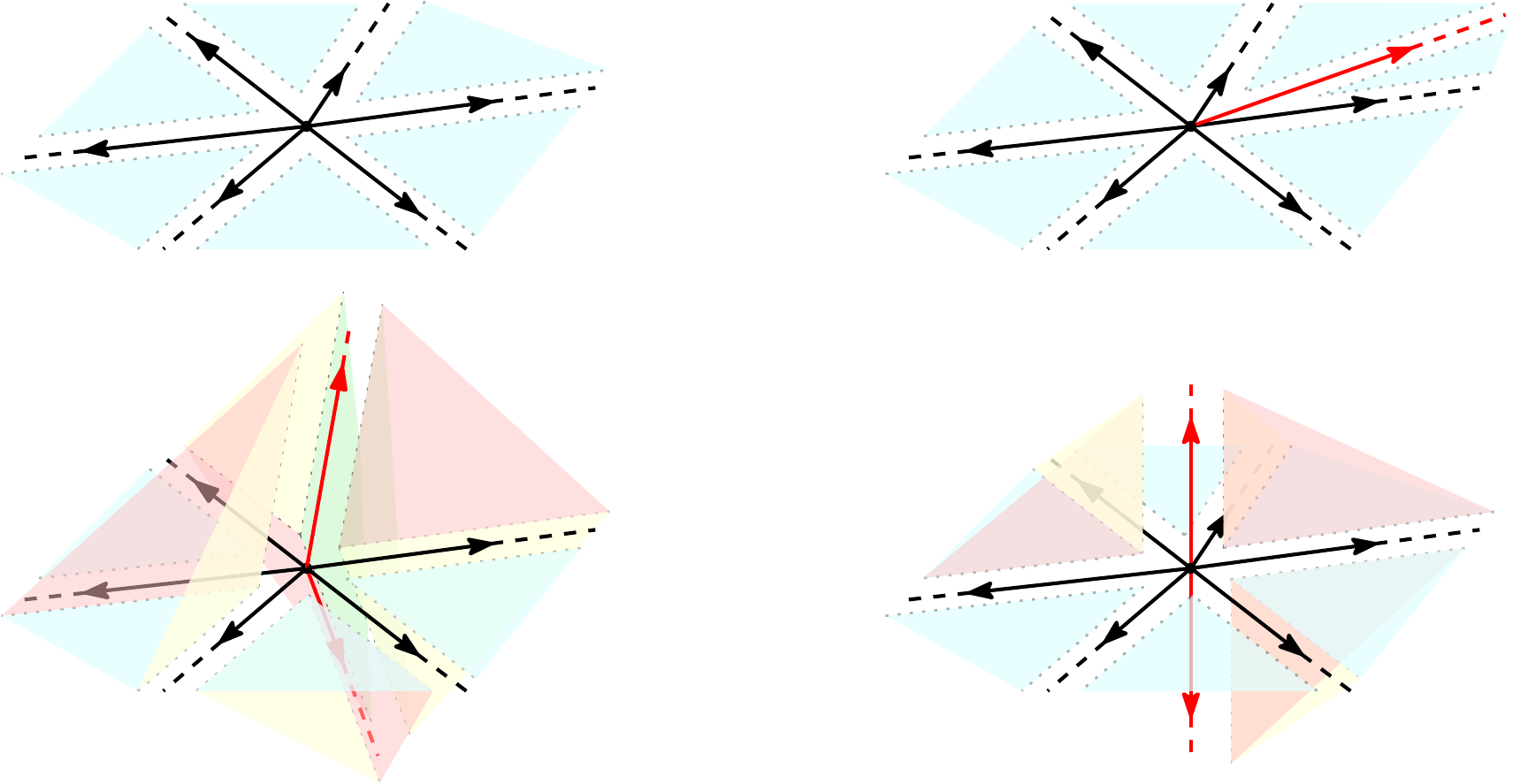}\put(87,205){$\ray[1]$}\put(113,182){$\ray[2]$}\put(115,203){$\face$}\put(-40,180){$\fan=$}\put(-55,60){$\ops_{\complex_\fan}(\ray[1])\cong$}\put(193,180){$\stell_{\complex_\fan}(\face)\cong$}\put(185,60){$\complex_\fan\join\{u_1,u_2\}\cong$}\end{overpic}}
\caption{A complete~$2$-dimensional simplicial fan~$\fan$ with a distinguished face~$\face=\cone\{\ray[1],\ray[2]\}$ (top left), a complete~$2$-dimensional simplicial fan realizing the complex~$\stell_{\complex_\fan}(\face)$ (top right), a complete~$3$-dimensional simplicial fan realizing the complex~$\ops_{\complex_\fan}(\ray[1])$ (bottom left), and a complete~$3$-dimensional simplicial fan realizing the suspension of the complex~$\complex_\fan$ (bottom right). The new fans are obtained by the generic transformations on~$\fan$ described after Lemma~\ref{lem:realizability}.}
\label{fig:operationsOnFans}
\end{figure}

\subsection{Subword complexes}
\label{subsec:subwordComplexes}

For~$n\ge1$, we denote the symmetric group of permutations of~$[n+1]$ by~$\symmetricGroup_{n+1}$, and by~$S$ the set of simple transpositions~$s_i\eqdef(i\,\,\,i+1)$ (for~$i\le n$), that we consider as an alphabet. To avoid confusion, a simple transposition will be referred to with an italic letter~$s_i$ when considered as an element of~$\symmetricGroup_{n+1}$, and with a san serif letter~$\sq{s}_i$ when considered as a letter in~$S$. Since simple transpositions generate~$\symmetricGroup_{n+1}$, any permutation~$\permutation\in\symmetricGroup_{n+1}$ can be written as a product~$\permutation =s_{i_1}\cdots s_{i_\ell}$. The word~$\sq{s}_{i_1}\dots\sq{s}_{i_\ell}$ is then called an~\defn{expression} of the permutation~$\permutation$. It is a~\defn{reduced expression} of~$\permutation$ if~$\ell$ is smallest possible among all expressions, in which case~$\ell$ is called the~\defn{length} of~$\permutation$. We denote by~$w_\circ\eqdef[n+1,\,n,\dots, 1]$ the unique longest element in~$\symmetricGroup_{n+1}$, also referred to as the~\defn{maximal permutation}.

Given a word~$\word=\sq{q}_{1}\dots\sq{q}_{p}$ in~$S^*$, a~\defn{subword} of~$\word$ is a subsequence~$\sq{q}_{i_1}\dots\sq{q}_{i_r}$ ($1\le i_1< \dots < i_r\le p$) of its letters. A~\defn{factor} of~$\word$ is a subword of~$\word$ consisting of consecutive letters and a~\defn{prefix} (resp.~\defn{suffix}) of~$\word$ is a factor containing its first (resp. last) letter. For any set~$J\subseteq[p]$, we denote by~$\word_{J}$ the subword of~$\word$ consisting of the letters with index in~$J$. If~$\word$ contains a reduced expression of~$w_\circ$ as a subword, the~\defn{subword complex}~$\subwordComplex(\word)$ (see~\cite{KnutsonMiller-subwordComplexesCoxeterGroups}) is the simplicial complex defined by\footnote{This is in fact a definition of spherical subword complexes of type~$A_n$.}
\[
\subwordComplex(\word)\eqdef\{J\subseteq[p]\,|\,\word_{[p]\ssm J}\text{ is a reduced expression of }w_\circ\}.
\]

We always consider a letter~$\sq{q}_r$ in a word~$\word$ as both data of its~\defn{position}~$r$ in~$\word$ and of the actual letter~$\sq{s}_i$ in the alphabet~$S$ such that~$\sq{q}_r=\sq{s}_i$. We identify the~\defn{vertices} of the subword complex~$\subwordComplex(\word)$ to the letters of~$\word$ whose position is contained in a facet of the complex and denote their set by~$\vertices_{\word}$. Observe that the vertices of~$\subwordComplex(\word)$ are the letters of~$\word$ which are not contained in all reduced expressions of~$w_\circ$ contained in~$\word$. The other letters of the word~$\word$ are the~\defn{non-vertices} of~$\subwordComplex(\word)$.
A convenient way to think about a subword complex~$\subwordComplex(\word)$ consists in encoding the underlying word~$\word=\sq{q}_1\dots\sq{q}_p$ with a set of segments~$\network_{\word}\eqdef\{\segment_r\,|\,r\in[p]\}$ called its corresponding~\defn{sorting network}~\cite{PilaudPocchiola,PilaudSantos-brickPolytope}. Each letter~$\sq{q}_r=\sq{s}_i$ ($r\in[p],i\in[n]$) is represented by a vertical segment~$\segment_r$ whose extremities have respective~$y$-coordinate~$i$ and~$i+1$ in~$\Rmbb^2$. Moreover if two letters~$\sq{q}_r=\sq{s}_i$ and~$\sq{q}_{r'}=\sq{s}_j$ (with~$r<r'$) satisfy~$|i-j|\le1$, then the points in~$\segment_{r'}$ have greater~$x$-coordinate than these in~$\segment_r$ (see Figure~\ref{fig:subwordComplexSortingNetwork} for an example). Our definition does not determine the segments, but we consider the sets of segments satisfying the conditions up to combinatorial equivalence and call any of them~\emph{the} sorting network of~$\subwordComplex(\word)$. The main property of the sorting network~$\network_{\word}$ is that it fully describes the combinatorics of the subword complex~$\subwordComplex(\word)$ even if it does not allow to recover the initial order on the letters of~$\word$. Indeed since~$s_i$ and~$s_j$ commute when~$|i-j|\ge2$, if~$\word'$ is obtained from~$\word$ by replacing a factor~$\sq{s}_i\sq{s}_j$ by~$\sq{s}_j\sq{s}_i$, then the subword complexes~$\subwordComplex(\word)$ and~$\subwordComplex(\word')$ clearly are isomorphic.

\begin{figure}
\centerline{\includegraphics[width=\textwidth]{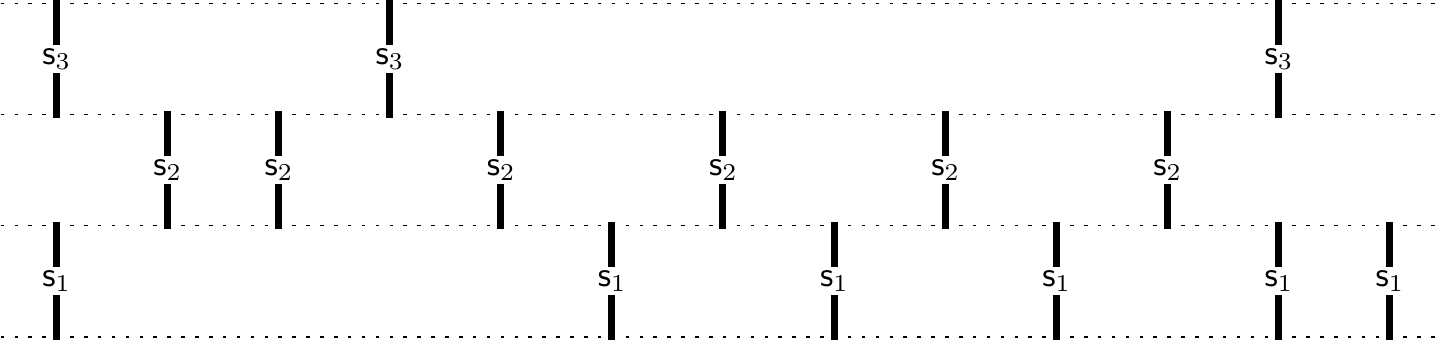}}
\caption{The sorting network~$\network_{\word}$ of the subword complex~$\subwordComplex(\word)$ for~$n=3$ and~$\word = \sq{s}_1\sq{s}_3\sq{s}_2\sq{s}_2\sq{s}_3\sq{s}_2\sq{s}_1\sq{s}_2\sq{s}_1\sq{s}_2\sq{s}_1\sq{s}_2\sq{s}_3\sq{s}_1\sq{s}_1$.}
\label{fig:subwordComplexSortingNetwork}
\end{figure}

Let~$\sq{c}\eqdef\sq{s}_1\dots\sq{s}_n$, the~\defn{$\sq{c}$-sorted expression} of~$w_\circ$ (see~\cite{Reading-cambrianLattices}) is the word
\begin{equation}
\sq{w}_\circ(\sq{c})\eqdef\prod_{i=1}^n\,\,\left(\prod_{j=1}^{n+1-i}\sq{s}_{j}\right)
\label{eq:coordinatesW}
\end{equation}

\noindent where the product denotes the concatenation on words in increasing order of indices, that is~$\sq{w}_\circ(\sq{c})=\sq{s}_1\sq{s}_2\dots\sq{s}_n\,\,\,\,\sq{s}_1\dots\sq{s}_{n-1}\,\,\,\,\dots\,\,\,\,\sq{s}_1\sq{s}_2 \,\,\,\,\sq{s}_1$. We will also use the notation~$\sq{c}[i]\eqdef\sq{s}_1\dots\sq{s}_i$, so that~$\sq{w}_\circ(\sq{c})=\prod_{i=1}^{n}\sq{c}[n+1-i]$. The $\sq{c}$-sorted expression~$\sq{w}_\circ(\sq{c})$ is a reduced expression of~$w_\circ$. It will be convenient to consider smaller symmetric groups~$\symmetricGroup_p$ (with~$p\le n+1$) as embedded in~$\symmetricGroup_{n+1}$ and still denote by~$\sq{w}_\circ(\sq{c}[p])$ the~$\sq{c}$-sorted expression of their longest element. The following statement describes how multiassociahedra arise as instances of subword complexes.

\begin{theorem}[\cite{PilaudPocchiola,Stump}]
\label{thm:identificationMultiassociahedronSubwordComplex}
For~$k\ge0$ and~$n\ge1$, the multiassociahedron~$\multiassociahedron_{k,n}$ is isomorphic to the subword complex~$\subwordComplex(\sq{c}^k\sq{w}_\circ(\sq{c}))$. Given a convex polygon~$\polygon_{n+2k+1}$ with~$n+2k+1$ vertices cyclically labeled from~$1$ to~$n+2k+1$, an isomorphism between the complexes~$\multiassociahedron_{k,n}$ and~$\subwordComplex(\sq{c}^k\sq{w}_\circ(\sq{c}))$ is given by the following identification of the~$k$-relevant diagonals of~$\polygon_{n+2k+1}$ with the letters of~$\sq{c}^k\sq{w}_\circ(\sq{c})$.
\begin{compactitem}
 \item for~$i\le k$ and~$j\in[n]$, the diagonal~$(i,i+j+k)$ is associated to the letter at position~$(i-1)n + j$ in the word~$\sq{c}^k\sq{w}_\circ(\sq{c})$, namely its~$i$-th letter~$\sq{s}_j$;
 \item for~$i+k\ge k+1$ and~$j\in[n + 1 - i]$, the diagonal~$(i+k,i+j+2k)$ is associated to the letter in position~$(k+i-1)n - (i-1)(i-2)/2 + j$ in the word~$\sq{c}^k\sq{w}_\circ(\sq{c})$, namely the letter~$\sq{s}_j$ whose indices are~$i$ and~$j$ in the factor~$\sq{w}_\circ(\sq{c})$ of the word~$\sq{c}^k\sq{w}_\circ(\sq{c})$, seen as the product in Equation~\eqref{eq:coordinatesW}.
\end{compactitem}
\end{theorem}

The identification given in Theorem~\ref{thm:identificationMultiassociahedronSubwordComplex} is illustrated in Figure~\ref{fig:identificationMultiassociahedronSubwordComplex}. We conclude this section with some pleasant properties of subword complexes. Let~${\word=\sq{s}_{i_1}\dots\sq{s}_{i_\ell}}$ be a word, the~\defn{rotated word} of~$\word$ is the word~$\word^\rotated:=\sq{s}_{n+1-i_\ell}\sq{s}_{i_1}\,\dots\,\sq{s}_{i_{\ell-1}}$.

\begin{figure}
\centerline{\includegraphics[width=.9\textwidth]{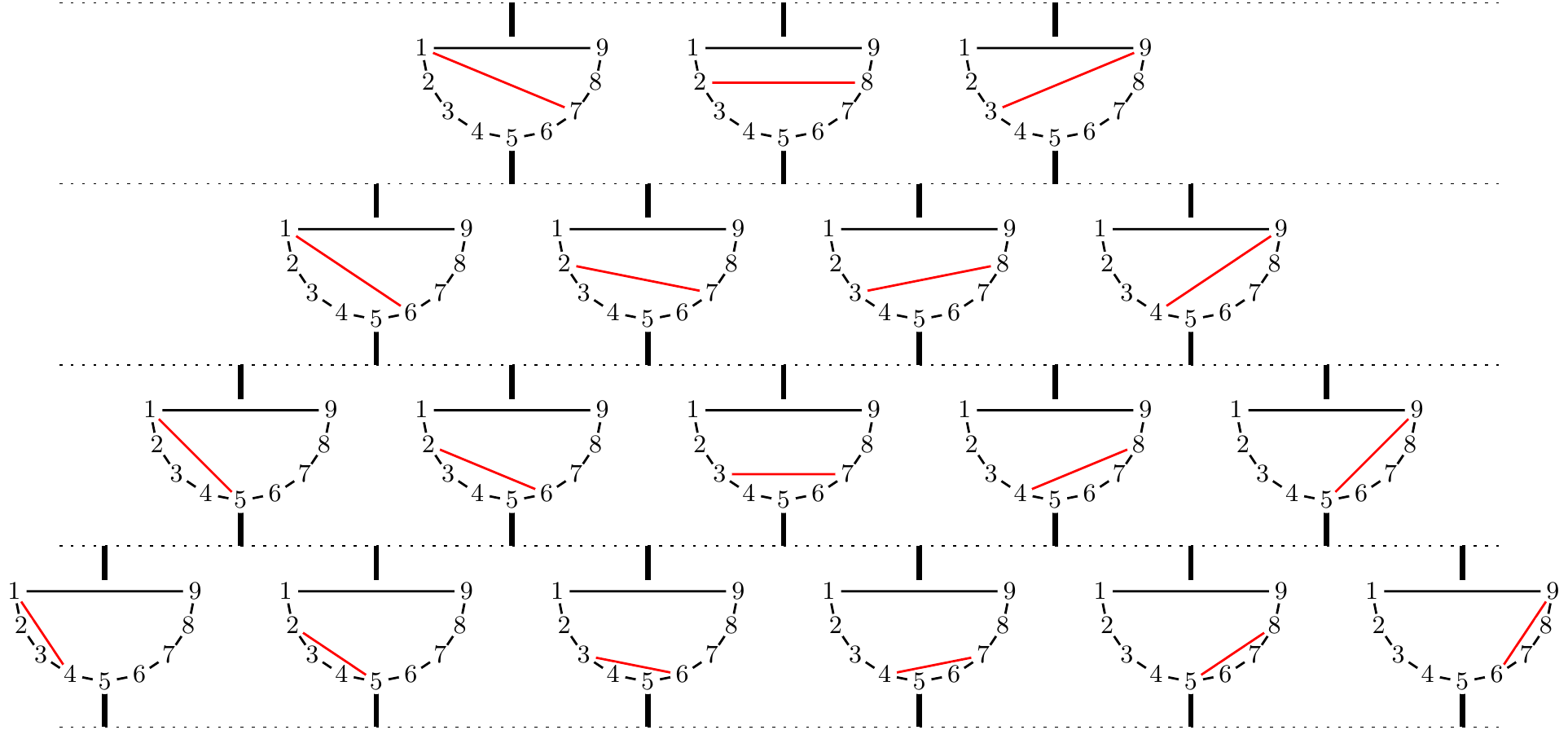}}
\caption{The identification of Theorem~\ref{thm:identificationMultiassociahedronSubwordComplex}, for~$n=4$ and~$k=2$, between the~$2$-relevant diagonals (of length at least~$2$) of a convex polygon with~$9$ vertices and the letters of the word~$\sq{c}^2\sq{w}_\circ(\sq{c})$ seen on the corresponding sorting network.}
\label{fig:identificationMultiassociahedronSubwordComplex}
\end{figure}

\begin{theorem}[rotation map~\cite{CeballosLabbeStump}]
\label{thm:rotationMap}
For any word~$\word$, the subword complexes~$\subwordComplex(\word)$ and~$\subwordComplex(\word^\rotated)$ are isomorphic. An isomorphism is obtained by identifying all letters in the common factor of~$\word$ and~$\word^\rotated$, and the two letters by which they differ.
\end{theorem}

Let~$\word\transpose$ denote the mirror image of a word~$\word$ and~$\word^{-i}$ the concatenation of~$i$ copies of~$\word\transpose$. Theorems~\ref{thm:identificationMultiassociahedronSubwordComplex} and~\ref{thm:rotationMap} imply that the multiassociahedron is isomorphic to all complexes of the form~$\subwordComplex(\sq{c}^{k-i}\sq{w}_\circ(\sq{c})\sq{c}^{-i})$ (for~$i\in[k]$). Finally basic properties of sorting networks imply that any subword complex~$\subwordComplex(\word)$ is isomorphic to~$\subwordComplex(\word\transpose)$ with identification of the vertices given by the mirror symmetry (see~\cite{PilaudPocchiola,PilaudSantos-brickPolytope}).

\subsection{Operations on subword complexes}
\label{subsec:operationsSubwordComplexes}

We now focus on three natural operations on subword complexes. A word~$\word$ is~\defn{obtained by a commutation move} from a word~$\word'$ if there exists two words~$\factor_1,\factor_2\in S^*$ and~$i,j\in[n]$ such that~$|i-j|\ge2$ and~$\word = \factor_1\, \sq{s}_i\sq{s}_j\, \factor_2$ and~$\word' = \factor_1 \,\sq{s}_j\sq{s}_i\, \factor_2$. As mentioned in Section~\ref{subsec:subwordComplexes} in this case the subword complexes~$\subwordComplex(\word)$ and~$\subwordComplex(\word')$ clearly are isomorphic since~$s_i$ and~$s_j$ commute in the symmetric group~$\symmetricGroup_n$. We also consider two operations studied by M.~Gorsky in~\cite{Gorsky-nilHeckeMoves,Gorsky-braidMoves}, which will be our main combinatorial tools.
\begin{compactitem}
 \item A word~$\word$ is~\defn{obtained by a~$0$-Hecke move} from a word~$\word'$ if there exists two words~$\factor_1,\factor_2\in S^*$ and~$i\in[n]$ such that~$\word = \factor_1\, \sq{s}_i\, \factor_2$ and~$\word' = \factor_1 \,\sq{s}_i^2\, \factor_2$. In this case~$\word'$ is~\defn{obtained by a reverse~$0$-Hecke move} from~$\word$. Alternatively we will also say that~$\word'$ is~\defn{obtained by doubling a letter} in~$\word$.
 \item A word~$\word$ is~\defn{obtained by a braid move} from a word~$\word'$ if there exists two words~$\factor_1,\factor_2\in S^*$ and~$i,j\in[n]$, with~$|i-j|=1$, such that~$\word = \factor_1\, \sq{s}_i\sq{s}_j\sq{s}_i\, \factor_2$ and~$\word' = \factor_1 \,\sq{s}_j\sq{s}_i\sq{s}_j\, \factor_2$.
\end{compactitem}

\begin{remark}
\label{rem:moves}
Commutation and braid moves are natural operations to consider since the corresponding relations~$s_is_j=s_js_i$ for~$|i-j|\ge2$, and~$s_is_js_i=s_js_is_j$ for~$|i-j|=1$, hold in the symmetric group~$\symmetricGroup_n$, and can be completed into a presentation of~$\symmetricGroup_n$ by adding the relations~$s_i^2=\id$ for~$i\in[n]$. Here the corresponding operation on words is replaced by~$s_i^2=s_i$, which is in fact the last relation in the classical presentation of the~$0$-Hecke algebra of the symmetric group~$\symmetricGroup_n$, hence the name of the corresponding transformation. M.~Gorsky calls these moves~\emph{nil-Hecke moves} in~\cite{Gorsky-nilHeckeMoves} but the corresponding relation in the nil-Hecke algebra would be~$s_i^2=0$.
\end{remark}

We say that we~\defn{apply a braid} (resp.~\defn{$0$-Hecke}, resp.~\defn{commutation})~\defn{move} to a subword complex~$\subwordComplex(\word)$ when we consider the subword complex~$\subwordComplex(\word')$, where~$\word$ and~$\word'$ are related by the same operation. The combinatorial effect of these operations on the subword complex~$\subwordComplex(\word)$ depend on the vertex status of the letters implied in the transformation (see Figure~\ref{fig:sumUpMoves}) and were described by M.~Gorsky as follows.

\begin{figure}
\centerline{\begin{overpic}[width=.73\textwidth]{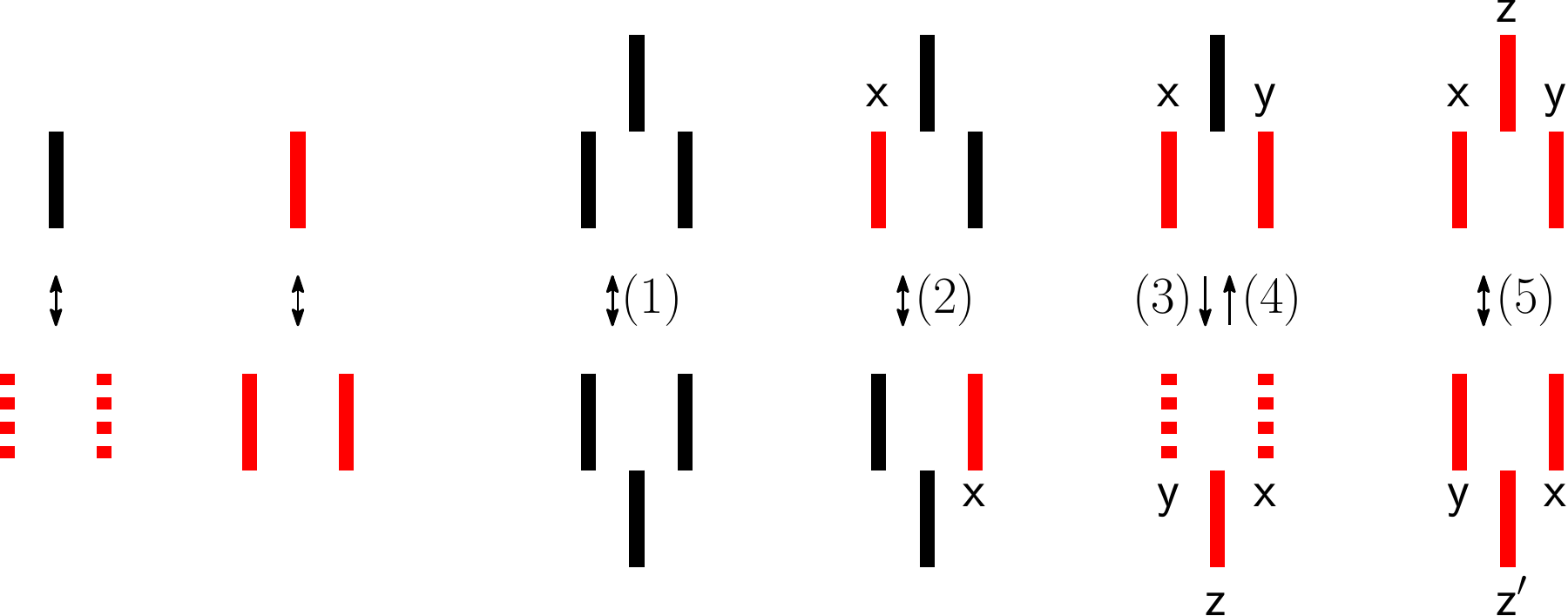}\put(1,-12){$0$-Hecke move}\put(160,-12){Braid move}\end{overpic}}
\vspace{.3cm}
\caption{The evolution of the vertex status of letters implied in~$0$-Hecke and braid moves, seen on the corresponding part of the sorting network~$\network_\word$. Two red dashed segments denote vertices not belonging to a common edge while empty red segment denote vertices in all faces not forbidden by dashed segments. Black plain segments represent non-vertices. The letters~$\sq{x},\sq{y},\sq{z}$ and~$\sq{z}'$ give the identifications of the exchanged letters by a braid move in the topological realizations of the corresponding complexes. The numbers for braid moves correspond to the different Cases in Theorem~\ref{thm:effectBraidMove}.}
\label{fig:sumUpMoves}
\end{figure}

\begin{theorem}[\cite{Gorsky-nilHeckeMoves}]
\label{thm:effectNilHeckeMove}
Suppose that a word~$\word'=\factor_1\,\sq{q}'_r\sq{q}'_{r+1}\,\factor_2$ is obtained from a word~$\word=\factor_1\,\sq{q}_r\,\factor_2$ by doubling the letter~$\sq{q}_r$. If~$\sq{q}_r$ is a vertex of the subword complex~$\subwordComplex(\word)$, then the subword complex~$\subwordComplex(\word')$ is isomorphic to the one-point-suspension~$\ops_{\subwordComplex(\word)}(\sq{q}_r)$ of~$\subwordComplex(\word)$ with respect to~$\sq{q}_r$. Otherwise~$\subwordComplex(\word')$ is isomorphic to the suspension of~$\subwordComplex(\word)$. The suspension vertices in~$\subwordComplex(\word')$ are~$\sq{q}'_r$ and~$\sq{q}'_{r+1}$.
\end{theorem}

\begin{theorem}[\cite{Gorsky-braidMoves}]
\label{thm:effectBraidMove}
Suppose that a word~$\word'=\factor_1\,\sq{q}'_r\sq{q}'_{r+1}\sq{q}'_{r+2}\,\factor_2$ is obtained from a word~$\word=\factor_1\,\sq{q}_r\sq{q}_{r+1}\sq{q}_{r+2}\,\factor_2$ by applying a braid move. 
\begin{compactenum}
  \item If none of the letters~$\sq{q}_r,\sq{q}_{r+1}$ and~$\sq{q}_{r+2}$ is a vertex of~$\subwordComplex(\word)$, then~$\subwordComplex(\word)$ and~$\subwordComplex(\word')$ have the same vertices and are isomorphic.
 \item If exactly one of the letters~$\sq{q}_r,\sq{q}_{r+1}$ and~$\sq{q}_{r+2}$ is a vertex of~$\subwordComplex(\word)$, then it is either~$\sq{q}_r$ or~$\sq{q}_{r+2}$, say~$\sq{q}_r$ without loss of generality\footnote{Since~$\subwordComplex(\word)$ and~$\subwordComplex(\word\transpose)$ are isomorphic.}. In this case~$\subwordComplex(\word)$ and~$\subwordComplex(\word')$ are isomorphic and an isomorphism is given by identifying their common vertices and associating~$\sq{q}_r$ to~$\sq{q}'_{r+2}$. In particular~$\sq{q}'_r$ and~$\sq{q}'_{r+1}$ are non-vertices in~$\subwordComplex(\word')$.
 \item If exactly two of the letters~$\sq{q}_r,\sq{q}_{r+1}$ and~$\sq{q}_{r+2}$ are vertices of~$\subwordComplex(\word)$, then these letters are~$\sq{q}_r$ and~$\sq{q}_{r+2}$. If moreover~$\sq{q}_r\sq{q}_{r+2}$ is an edge of the complex~$\subwordComplex(\word)$, then the complex~$\subwordComplex(\word')$ is isomorphic to the stellar subdivision~$\stell_{\subwordComplex(\word)}(\sq{q}_r\sq{q}_{r+2})$ of the edge~$\sq{q}_r\sq{q}_{r+2}$ in the complex~$\subwordComplex(\word)$. In~$\subwordComplex(\word')$, the subdivision vertex is~$\sq{q}'_{r+1}$, the vertex~$\sq{q}'_r$ (resp.~$\sq{q}'_{r+2}$) is identified to the vertex~$\sq{q}_{r+2}$ (resp.~$\sq{q}_r$) of~$\subwordComplex(\word)$, and all common vertices are identified.
 \item If all letters ~$\sq{q}_r,\sq{q}_{r+1}$ and~$\sq{q}_{r+2}$ are vertices of~$\subwordComplex(\word)$ but do not all belong to a same facet, and if~$\sq{q}_r\sq{q}_{r+1}$ is an edge of the subword complex~$\subwordComplex(\word)$, then the roles of~$\subwordComplex(\word)$ and~$\subwordComplex(\word')$ are exchanged in the previous case. That is~$\subwordComplex(\word)$ is obtained from~$\subwordComplex(\word')$ by applying a stellar subdivision of the edge~$\sq{q}'_r\sq{q}'_{r+2}$.
 \item If all letters~$\sq{q}_r,\sq{q}_{r+1}$ and~$\sq{q}_{r+2}$ are vertices of~$\subwordComplex(\word)$ and belong to a common facet, then it is also the case for~$\sq{q}'_r,\sq{q}'_{r+1}$ and~$\sq{q}'_{r+2}$ in~$\subwordComplex(\word)$, and the two stellar subdivisions~$\stell_{\subwordComplex(\word)}(\sq{q}_r\sq{q}_{r+2})$ and~$\stell_{\subwordComplex(\word')}(\sq{q}'_r\sq{q}'_{r+2})$ are isomorphic. The vertex~$\sq{q}_r$ (resp.~$\sq{q}_{r+2}$) of the complex~$\stell_{\subwordComplex(\word)}(\sq{q}_r\sq{q}_{r+2})$ is identified to the vertex~$\sq{q}'_{r+2}$ (resp.~$\sq{q}'_r$) of the complex~$\stell_{\subwordComplex(\word')}(\sq{q}'_r\sq{q}'_{r+2})$, the subdivisions vertex in the complex~$\stell_{\subwordComplex(\word)}(\sq{q}_r\sq{q}_{r+2})$ (resp.~$\stell_{\subwordComplex(\word')}(\sq{q}'_r\sq{q}'_{r+2})$) is identified to~$\sq{q}'_{r+1}$ (resp.~$\sq{q}_{r+1}$) and all common vertices are identified.
\end{compactenum}
\end{theorem}

\begin{figure}
\centerline{\includegraphics[width=1.25\textwidth]{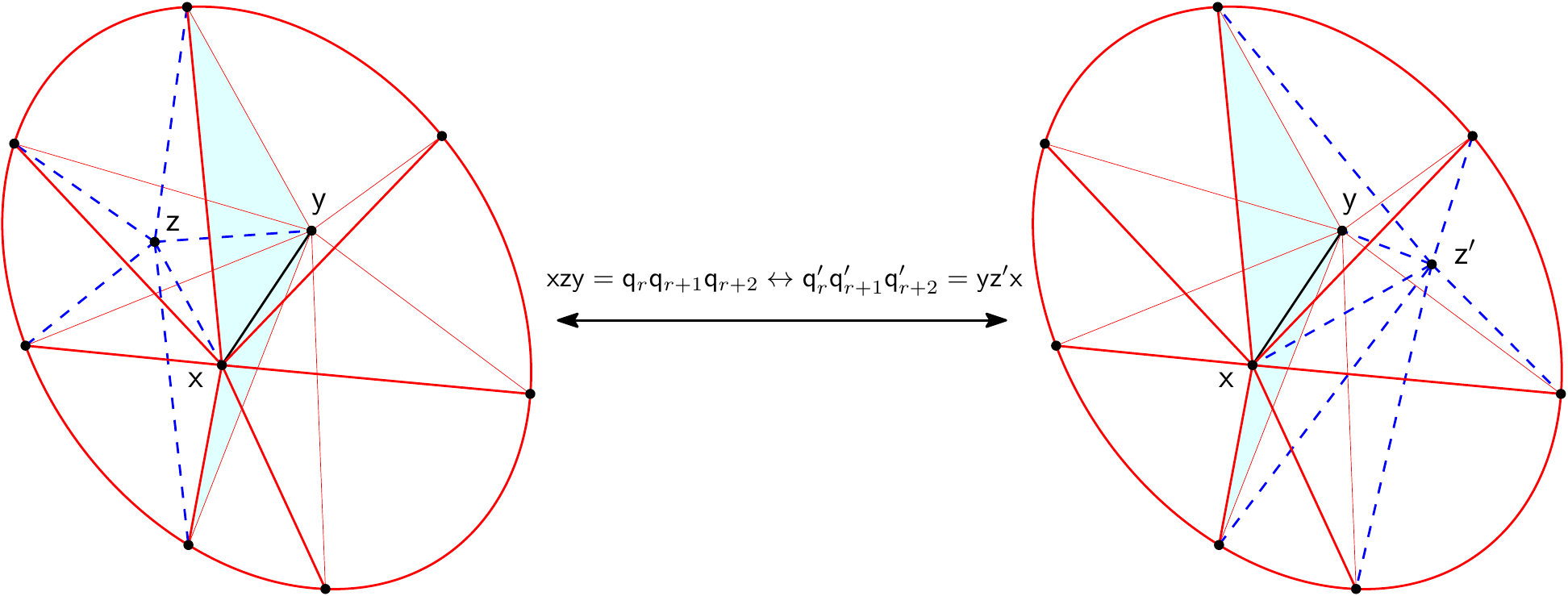}}
\caption{The effect of a braid move in Case~(5) of Theorem~\ref{thm:effectBraidMove}, seen as a local transformation on the topological realization of a~$3$-dimensional subword complex~$\subwordComplex(\word)$. Only the parts of the complex affected by the transformation are depicted, namely the star~$\st_{\subwordComplex(\word)}(\sq{q}_{r+1})$ of the vertex~$\sq{q}_{r+1}$ in~$\subwordComplex(\word)$ and the star~$\st_{\subwordComplex(\word)}(\sq{q}_r\sq{q}_{r+2})$ of the edge~$\sq{q}_r\sq{q}_{r+2}$ in~$\subwordComplex(\word)$. As in Figure~\ref{fig:sumUpMoves}, the letters~$\sq{x},\sq{y},\sq{z}$ and~$\sq{z}'$ describe the identifications of the exchanged letters in the topological realizations of the two subword complexes.}
\label{fig:topologicalBraidMove}
\end{figure}

\begin{remark}
Theorem~\ref{thm:effectBraidMove} does not seem to cover all possible cases because of the additional condition that the edge~$\sq{q}_r\sq{q}_{r+2}$ (resp~$\sq{q}_r\sq{q}_{r+1}$) exists in Case~(3) (resp.~(4)). In fact this condition is always satisfied. Indeed the~\emph{word property} (see~\cite[Theorem 3.3.1]{BjornerBrenti}) asserts that any expression can be transformed into a reduced one of the corresponding element by a sequence using only braid moves and simplifications~$s^2=\id$. Now one can show inductively that in any type~$A_n$ spherical subword complex~$\subwordComplex(\word)$, any two letters~$\sq{q}_\ell,\sq{q}_r$ of the word~$\word$ which are vertices of the complex form an edge of the complex or are exchangeable (maybe both) as follows. The property is clear when~$\word$ is already a reduced expression of~$w_\circ$ since the complex is then empty. Now using the word property, one can derive, for any word~$\word$ containing an expression for~$w_\circ$, a sequence of braid,~$0$-Hecke and commutation moves to move it to a reduced expression of~$w_\circ$. Then one can do the proof inductively on the minimal length of such a sequence, applying the induction hypothesis and Theorems~\ref{thm:effectNilHeckeMove} and~\ref{thm:effectBraidMove}. The induction hypothesis is indeed preserved by reverse one-point-suspension and basic arguments on sorting networks imply that it is also preserved by braid moves. Finally it is also standard in this framework to check that the letters~$\sq{q}_r$ and~$\sq{q}_{r+2}$ (resp.~$\sq{q}_{r+1}$) are not exchangeable in Case~(3) (resp~(4)) of Theorem~\ref{thm:effectBraidMove} (see~\cite{PilaudSantos-brickPolytope}), so that~$\sq{q}_r\sq{q}_{r+2}$ (resp~$\sq{q}_r\sq{q}_{r+1}$) is in fact automatically an edge of the complex. Figure~\ref{fig:sumUpMoves} sums up all cases of Theorems~\ref{thm:effectNilHeckeMove} and~\ref{thm:effectBraidMove} in a more visual way than with their actual descriptions.
\end{remark}

The effects of~$0$-Hecke and braid moves are already illustrated on the complex itself in Figures~\ref{fig:onePointSuspension} and~\ref{fig:stellarSubdivision}, except for Case~(5) of Theorem~\ref{thm:effectBraidMove}. Indeed we have described a relation between the subword complexes~$\subwordComplex(\word)$ and~$\subwordComplex(\word')$, but not in terms of ``transforming~$\subwordComplex(\word)$ into~$\subwordComplex(\word')$''. In this case~$\subwordComplex(\word')$ is obtained from~$\subwordComplex(\word)$ by a stellar subdivision of the edge~$\sq{q}_r\sq{q}_{r+2}$ with subdivision vertex~$\sq{q}'_{r+1}$, followed by a reverse stellar subdivision of the same edge where the disappearing vertex is~$\sq{q}_{r+1}$. This can somehow be geometrically interpreted as ``moving'' the vertex~$\sq{q}_{r+1}$ ``from one side'' of the edge~$\sq{q}_r\sq{q}_{r+2}$ ``to the other side'' and relabeling it~$\sq{q}'_{r+1}$ (see Figure~\ref{fig:topologicalBraidMove}).
Reverse stellar subdivision is bad behaved with respect to geometry, and we present in the next sections a tentative  construction of~$2$-associahedra based on commutation moves,~$0$-Hecke moves and braid moves avoiding Case~(4) of Theorem~\ref{thm:effectBraidMove}, and as much as possible Case~(5) of Theorem~\ref{thm:effectBraidMove}. The key point is that the effect induced by moves on the complex both only depends on the local data of the vertex status of the implied letters, and is itself topologically local.

\section{Loday associahedron by suspensions and stellar subdivisions}
\label{sec:associahedra}

We keep the product notation~$\prod_{j=\mu}^{\nu}\word_j$ to denote the concatenation of a sequence of words~$(\word_j)_{\mu\le j\le\nu}$, with the convention that empty products represent the empty word. We will use the notation~$\sq{c}[\mu,\nu]=\prod_{j = \mu}^{\nu}\sq{s}_j$ for~$1\le \mu \le \nu\le n$. Recall that with this notation, we have~$\sq{c}[\mu]=\sq{c}[1,\mu]$ for~$\mu\in\Nmbb$ and~$\sq{w}_\circ(\sq{c})=\prod_{i=1}^n\sq{c}[n+1-i]$. Moreover in any word, we will from now on denote the positions of the letters of a distinguished factor~$\sq{w}_\circ(\sq{c})$ by the corresponding pair~$(i,j)$ of indices in the double product formula~$\sq{w}_\circ(\sq{c})\stackrel{\raisebox{.04cm}{\text{\tiny Eq~\eqref{eq:coordinatesW}}}}{\raisebox{-.04cm}{=}}\prod_{i=1}^n\prod_{j=1}^{n+1-i}\sq{s}_j$. A word~$\word$ can be~\defn{moved to} a word~$\word'$ if~$\word$ can be transformed into~$\word'$ by applying a sequence of commutation,~$0$-Hecke, reverse~$0$-Hecke and braid moves. If~$\word$ can be moved to~$\word'$ using only commutations moves, then the words~$\word$ and~$\word'$ are~\defn{equivalent under commutation} (or simply~\defn{equivalent}) and we use the notation~$\word\commuteTo\word'$. Notice that in this case, the two subword complexes~$\subwordComplex(\word)$ and~$\subwordComplex(\word')$ are isomorphic.

\begin{lemma}
\label{lem:insertionLetter}
For~$\ell\in[n]$, the~$\sq{c}$-sorted word~$\sq{w}_\circ(\sq{c})$ can be moved to the word~$\sq{w}_\circ(\sq{c})\sq{s}_\ell$ by doubling its letter~$\sq{s}_1$ at position~$(\ell,1)$ (that is its~$\ell$-th letter~$\sq{s}_1$) and applying a sequence of~$\ell-1$ braid moves interlaced with some commutation moves.
\end{lemma}

\begin{proof}
For~$k\in[\ell-1]$, we can let the second letter~$\sq{s}_k$ of the word~$\sq{s}_k\sq{c}[k+1,\ell]\sq{c}[k,\ell-1]$ commute as much as possible to the left to obtain that it is equivalent to the word~$\sq{s}_k\sq{s}_{k+1}\sq{s}_k\sq{c}[k+2,\ell]\sq{c}[k+1,\ell-1]$. We can then apply a braid move on the prefix~$\sq{s}_k\sq{s}_{k+1}\sq{s}_k$ of this last word to obtain the word~${\sq{s}_{k+1}\sq{s}_k\sq{s}_{k+1}\sq{c}[k+2,\ell]\sq{c}[k+1,\ell-1]}$. Therefore a straightforward induction on~$k$ shows that the word~$\sq{s}_1\sq{c}[\ell]\sq{c}[\ell-1]$ can be moved to the word~$\sq{c}[\ell]\sq{c}[\ell-1]\sq{s}_\ell$ applying~$\ell-1$ braid moves interlaced with some commutation moves. Multiplying both words by~$\sq{w}_\circ(\sq{c}[\ell-2])$ on the right, the identity~$\sq{w}_\circ(\sq{c}[\ell])=\sq{c}[\ell]\sq{c}[\ell-1]\sq{w}_\circ(\sq{c}[\ell-2])$ yields that the word~$\sq{s}_1\sq{w}_\circ(\sq{c}[\ell])$ can be moved to the word~$\sq{c}[\ell]\sq{c}[\ell-1]\sq{s}_\ell\sq{w}_\circ(\sq{c}[\ell-2])$ by a sequence of~$\ell-1$ braid moves interlaced with some commutation moves. Since the letter~$\sq{s}_\ell$ commutes to all letters of the word~$\sq{w}_\circ(\sq{c}[\ell-2])$, we obtain by the same identity that the word~$\sq{s}_1\sq{w}_\circ(\sq{c}[\ell])$ can be moved to the word~$\sq{w}_\circ(\sq{c}[\ell])\sq{s}_\ell$ by a sequence of~$\ell-1$ braid moves interlaced with some commutation moves. This is the result for~$\ell=n$ since the word~$\sq{s}_1\sq{w}_\circ(\sq{c}[\ell])$ is obtained from the word~$\sq{w}_\circ(\sq{c}[\ell])$ by doubling its first letter~$\sq{s}_1$. Since any word~$\sq{w}_\circ(\sq{c}[\ell])$ (for~$\ell\in[n]$) is a suffix of the word~$\sq{w}_\circ(\sq{c}[n])$, the result for any~$\ell\in[n]$ finally follows by multiplying on the left by the suitable prefix.
\end{proof}

Applying Lemma~\ref{lem:insertionLetter} repeatedly, we obtain the following more specific statement.

\begin{corollary}[fattening a triangle]
\label{cor:fatteningTriangle}
The~$\sq{c}$-sorted word~$\sq{w}_\circ(\sq{c})$ can be moved to the word~$\sq{w}_\circ(\sq{c})\sq{c}\transpose = \sq{w}_\circ(\sq{c})\, \sq{s}_n\, \sq{s}_{n-1}\, \dots \, \sq{s}_1$ by doubling all its letters~$\sq{s}_1$ and applying a sequence of~$n(n-1)/2$ braid moves interlaced with some commutation moves.
\end{corollary}

\begin{proof}
We check by induction on~$n\ge1$ that all the letters~$\sq{s}_1$ in the word~$\sq{w}_\circ(\sq{c})$ can be doubled at first before applying the other moves of Lemma~\ref{lem:insertionLetter}. The word~$\sq{s}_1\sq{w}_\circ(\sq{c})$ can be obtained from the word~$\sq{w}_\circ(\sq{c})$ by doubling its first letter~$\sq{s}_1$. We apply the induction hypothesis to the suffix~$\sq{w}_\circ(\sq{c}[n-1])$ of the word~$\sq{s}_1\sq{w}_\circ(\sq{c})=\sq{s}_1\sq{c}\,\sq{w}_\circ(\sq{c}[n-1])$ to find a sequence of moves starting by doubling all the letters~$\sq{s}_1$ in this factor and transforming the word~$\sq{s}_1\sq{w}_\circ(\sq{c})$ into the word~$\sq{s}_1\sq{w}_\circ(\sq{c})(\sq{c}[n-1])^{-1}$. We finally apply Lemma~\ref{lem:insertionLetter} to the prefix~$\sq{s}_1\sq{w}_\circ(\sq{c})$ of this last word, omitting the initial doubling, into moving it to the word~$\sq{s}_1\sq{w}_\circ(\sq{c})\sq{s}_n$. Since~$\sq{c}^{-1}=\sq{s}_n(\sq{c}[n-1])^{-1}$, we are done.
\end{proof}

We will now refer to a distinguished factor~$\sq{w}_\circ(\sq{c})$ in a word~$\word=\factor_1\,\sq{w}_\circ(\sq{c})\,\factor_2$ as a~\defn{triangle} in~$\word$, because of the shape of the corresponding sorting network. We say that we~\defn{fatten a triangle in}~$\word$ when we consider a word~$\word'=\factor_1\,\sq{w}_\circ(\sq{c})\sq{c}\transpose\,\factor_2$ obtained from~$\word$ by applying the sequence of moves of Corollary~\ref{cor:fatteningTriangle} to its distinguished triangle. Let~$\word=\word_1,\dots,\word_\ell=\word'$ be the successive words obtained in this sequence of moves, where we write~$\word_k=\factor_1\factor[T]_k\factor_2$ for~$k\in[\ell]$. Notice that~$\factor[T]_1=\sq{w}_\circ(\sq{c})$ and~$\factor[T]_\ell=\sq{w}_\circ(\sq{c})\sq{c}^{-1}$. The operation of fattening a triangle comes together with a natural correspondence between the letters in the word~$\word$ and those in the word~$\word'$. The letters in the common factors~$\factor_1$ and~$\factor_2$ are indeed naturally identified, and the letters of the middle factor~$\sq{w}_\circ(\sq{c})\sq{c}\transpose$ of~$\word'$ can be associated to these of the middle factor~$\sq{w}_\circ(\sq{c})$ of~$\word$ using the following labeling rules along moves.

\begin{compactitem}
 \item The letters in the distinguished factor~$\sq{w}_\circ(\sq{c})$ of the word~$\word$ are labeled with their position (given by a pair of indices) in~$\sq{w}_\circ(\sq{c})$ (see Figure~\ref{fig:labelingRules} left).
 \item After doubling a letter~$\sq{s}_1$ at position~$(i,1)$ in a distinguished factor~$\factor[T]_k$ of a word~$\word_k, (k\in[\ell-1])$, we label the two resulting letters~$\sq{s}_1$ in~$\word_{k+1}$ with~$(i,1)$ and~$(i,1)'$. Indeed Theorem~\ref{thm:effectNilHeckeMove} asserts that the new subword complex~$\subwordComplex(\word_{k+1})$ is somehow the same as the previous subword complex~$\subwordComplex(\word_k)$, but with two copies of its initial letters~$\sq{s}_1$ (see Figure~\ref{fig:labelingRules} middle).
 \item After a braid move on a factor~$\sq{q}_r\sq{q}_{r+1}\sq{q}_{r+2}$ of a word~$\word_k$~$(k\in[\ell-1])$ producing a factor~$\sq{q}'_r\sq{q}'_{r+1}\sq{q}'_{r+2}$ in the word~$\word_{k+1}$, we label the letter~$\sq{q}'_{r}$ (resp.~$\sq{q}'_{r+1}$, resp.~$\sq{q}'_{r+2}$) in the word~$\word_{k+1}$ with the same label as that of the letter~$\sq{q}_{r+2}$ (resp.~$\sq{q}_{r+1}$, resp.~$\sq{q}_r$) in the word~$\word_k$ (see Figure~\ref{fig:labelingRules} right). This corresponds to the identifications suggested by Theorem~\ref{thm:effectBraidMove} (see Figure~\ref{fig:sumUpMoves}).
 \end{compactitem}

\begin{figure}
\centerline{\includegraphics[width=1.3\textwidth]{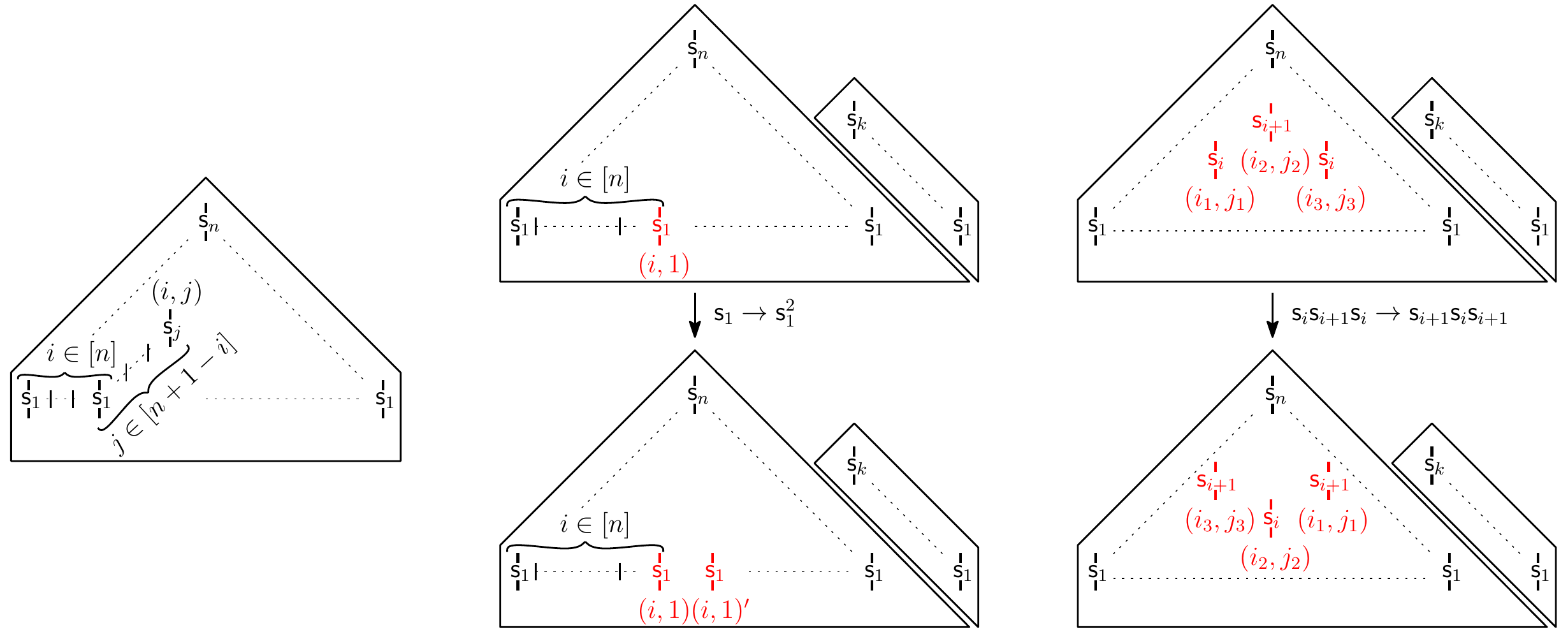}}
\caption{The evolution of the labels of the letters in a sequence of moves to fatten a triangle~$\sq{w}_\circ(\sq{c})$, seen on the corresponding sorting networks. The initial letters of the triangle~$\sq{w}_\circ(\sq{c})$ are labeled with their position~$(i,j)$ (with~$i\in[n],j\in[n+1-j]$) (left). After doubling a letter~$\sq{s}_1$ labeled~$(i,1)$ (top middle), the two new letters~$\sq{s}_1$ are labeled~$(i,1)$ and~$(i,1)'$ (bottom middle). The letters obtained by braid moves are labeled following the identification in Theorem~\ref{thm:effectBraidMove}~(right).}
\label{fig:labelingRules}
\end{figure}

Notice that even if there are cases in Theorem~\ref{thm:effectBraidMove}, the identification between the letters implied in a braid move always follows ours, independently of their vertex status (see Figure~\ref{fig:sumUpMoves}). The letters~$\sq{q}_{r}$ and~$\sq{q}'_{r+2}$ (resp~$\sq{q}_{r+2}$ and~$\sq{q}'_{r}$) are indeed always identified, and in each case the letter~$\sq{q}'_{r+1}$ is obtained by some transformation of the letter~$\sq{q}_{r+1}$. Figure~\ref{fig:refinementLemmaInsertionLetter} illustrates the labeling evolution rules on an example.

\begin{figure}
\centerline{\includegraphics[width=1.4\textwidth]{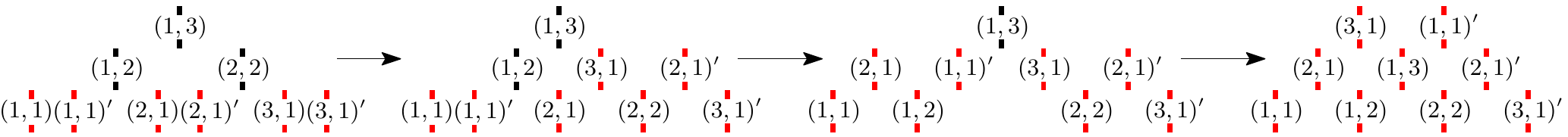}}
\vspace{-.05cm}
\caption{The vertex status and label evolution of letters along a sequence to fatten a triangle~$\sq{s}_1\sq{s}_2\sq{s}_3\sq{s}_1\sq{s}_2\sq{s}_1$, seen on the sorting networks of the words in the sequence. After doubling all letters~$\sq{s}_1$, one successively obtains the words~$\sq{s}_1\sq{s}_1\sq{s}_2\sq{s}_3\sq{s}_1\sq{s}_2\sq{s}_1\sq{s}_2\sq{s}_1\,,\,\sq{s}_1\sq{s}_2\sq{s}_1\sq{s}_2\sq{s}_3\sq{s}_2\sq{s}_1\sq{s}_2\sq{s}_1$ and~$\sq{s}_1\sq{s}_2\sq{s}_3\sq{s}_1\sq{s}_2\sq{s}_3\sq{s}_1\sq{s}_2\sq{s}_1$ by applying three braid moves and some commutation moves. Red empty (resp. black plain) segments denote vertices (resp. non-vertices) of the current subword complex.}
\label{fig:refinementLemmaInsertionLetter}
\vspace{-.25cm}
\end{figure}

\begin{lemma}
\label{lem:manyRefinementsLemma}
Let~$\word=\factor_1\sq{w}_\circ(\sq{c})\factor_2$ be a word with a distinguished triangle and let~$\word=\word_1,\dots,\word_\ell$ (with~$\word_k=\factor_1\factor[T]_k\factor_2$ for~$k\in[\ell]$) be a fattening sequence of this triangle.
\begin{compactitem}
 \item The labels of the letters of~$\factor[T]_\ell=\sq{w}_\circ(\sq{c})\sq{c}^{-1}=\sq{c}\sq{w}_\circ(\sq{c}[n-1])\sq{c}^{-1}$, obtained by the identification rules, are these of Figure~\ref{fig:identificationPattern}. Namely the letter at position~$i\in[n]$ in the~$\sq{c}$ prefix is labeled~$(i,1)$, the letter indexed~$(i,j)$ (with~$i\in[n-1],j\in[n-i]$) in the factor~$\sq{w}_\circ(\sq{c}[n-1])=\prod_{1\le i\le n-1}\prod_{1\le j\le n-i}\sq{s}_j$ is labeled~$(i,j+1)$ and the letter in position~$i\in[n]$ in the~$\sq{c}^{-1}$ suffix is labeled~$(n-i+1,1)'$.
 \item If the word~$\word_k$ (for~$k\in[\ell]$) contains a factor~$\sq{q}_r\sq{q}_{r+1}\sq{q}_{r+2}$ implied in a braid move, then the labels of~$\sq{q}_r,\sq{q}_{r+1}$ and~$\sq{q}_{r+2}$ in~$\word_k$ are~$(i,1)',(i,j+1)$ and~$(i+j,1)$ respectively for some~$i\in[n-1],j\in[n-i+1]$. Moreover~$\sq{q}_r\sq{q}_{r+2}$ is an edge of the subword complex~$\subwordComplex(\word_k)$, and~$\sq{q}_{r+1}$ is a vertex of~$\subwordComplex(\word_k)$ if and only if the letter with label~$(i,j+1)$ in~$\factor[T]_1$ is a vertex of~$\subwordComplex(\word)$.
\end{compactitem}
\end{lemma}

Lemma~\ref{lem:manyRefinementsLemma} translates obvious phenomena that can be observed on the example in Figure~\ref{fig:refinementLemmaInsertionLetter}. The proof is an easy but technical refinement of the proofs of Lemma~\ref{lem:insertionLetter} and Corollary~\ref{cor:fatteningTriangle}. We only give a sketch of it and leave the details to the reader.

\begin{proof}[Proof (sketch).]
The proof is by induction on~$n\ge1$, the case~$n=1$ being trivial. Given a triangle~$\sq{w}_\circ(\sq{c})=\sq{c}\sq{w}_\circ(\sq{c}[n-1])$ in a word~$\word$, we first fatten its subfactor~$\sq{w}_\circ(\sq{c}[n-1])$ into transforming~$\sq{w}_\circ(\sq{c})$ into the factor~$\sq{c}\sq{w}_\circ(\sq{c}[n-1])(\sq{c}(n-1])^{-1}$. In this last factor, the labels of the letters in the prefix~$\sq{c}$ are still the initial ones, while the labels of the other letters are described by the induction hypothesis. Moreover, the vertex status of the letters in the prefix~$\sq{c}$ is the same as in the word~$\word$ since the effect of reverse~$0$-Hecke and braid moves is local, by Theorems~\ref{thm:effectNilHeckeMove} and~\ref{thm:effectBraidMove}. To move the factor~$\sq{c}\sq{w}_\circ(\sq{c}[n-1])(\sq{c}[n-1])^{-1}=\sq{w}_\circ(\sq{c})(\sq{c}(n-1])^{-1}$ to the factor~$\sq{w}_\circ(\sq{c})\sq{c}^{-1}$, we apply Lemma~\ref{lem:insertionLetter} to its~$\sq{w}_\circ(\sq{c})$ prefix in order to insert a new letter~$\sq{s}_n$. It is then straightforward to adapt the induction in the proof of Lemma~\ref{lem:insertionLetter} into keeping track of the labels and vertex status of the letters in the final factor~$\sq{w}_\circ(\sq{c})\sq{c}^{-1}$, so as of the prescribed edges. The key point for the induction step is that doubling the first letter~$\sq{s}_1$ creates an edge between any of the two resulting letters and any other vertex~$\sq{q}$ of the current subword complex (by Theorem~\ref{thm:effectNilHeckeMove}), and that this edge is never affected by the stellar subdivisions and reverse stellar subdivisions corresponding to the braid moves (by Theorem~\ref{thm:effectBraidMove}) which do not imply both letters.
\end{proof}

\begin{figure}
\centerline{\includegraphics[width=1.3\textwidth]{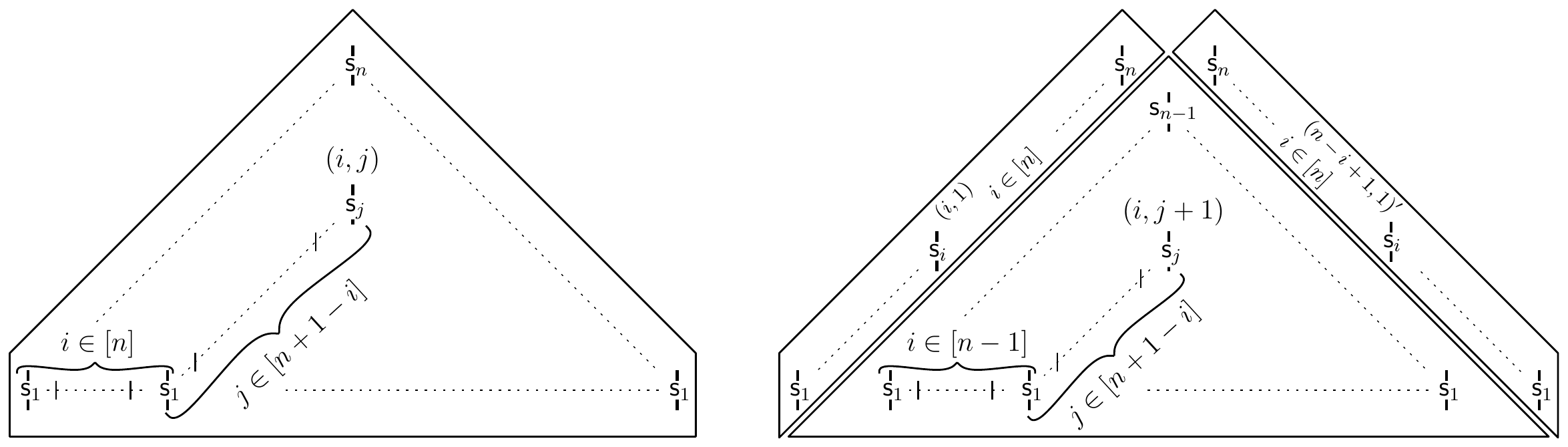}}
\caption{The identification pattern of the letters of a word~$\sq{w}_\circ(\sq{c})\sq{c}\transpose$ (right) obtained by fattening a triangle~$\sq{w}_\circ(\sq{c})$ (left) by following the labeling rules. The pattern is given by the labeling on the sorting networks of the two words.}
\label{fig:identificationPattern}
\vspace{-.4cm}
\end{figure}

Observe that Lemma~\ref{lem:manyRefinementsLemma} implies that all the braid moves implied in a fattening sequence either induce Case~(3) or Case~(5) of Theorem~\ref{thm:effectBraidMove}. Therefore Lemma~\ref{lem:manyRefinementsLemma} yields a new construction for the classical associahedron, choosing the word~$\word$ to be simply a triangle~$\word=\sq{w}_\circ(\sq{c})$. Indeed no letter is a vertex in~$\subwordComplex(\word)=\{\varnothing\}$, and so no letter with label~$(i,j+1)$ (for~$i\in[n-1],j\in[n-i+1]$) is a vertex in the subword complex~$\subwordComplex(\word')$, where the~$\word'$ is obtained from the word~$\word$ by doubling all its~$\sq{s}_1$ letters. Since those letters are non-vertices in the subword complex~$\subwordComplex(\word)$, all the corresponding reverse~$0$-Hecke moves induce suspensions by Theorem~\ref{thm:effectNilHeckeMove} so that the subword complex~$\subwordComplex(\word')$ is isomorphic to the boundary complex of the~$n$-dimensional cross polytope. Finally Lemma~\ref{lem:manyRefinementsLemma} implies that all braid moves in the fattening sequence induce stellar subdivisions of edges, by Theorem~\ref{thm:effectBraidMove}. We obtain by Lemma~\ref{lem:realizability} a construction of the classical associahedron by successive stellar subdivisions of edges of the cross-polytope.

\begin{corollary}
\label{cor:associahedronByTruncations}
The~$n$-dimensional simplicial associahedron can be obtained by successive stellar subdivisions of edges of the~$n$-dimensional cross-polytope. Equivalently its polar dual can be obtained by successive truncations of codimension-$2$ faces of the~$n$-dimensional cube.
\end{corollary}

This is a special case of a result by V.~Volodin~\cite{Volodin} stating that any flag nestohedron can be obtained by successive such truncations of a cube. The last figure in~\cite{Volodin} depicts a~$3$-dimensional associahedron geometrically equivalent to the realization by J.-L.~Loday~\cite{Loday} (see Figure~\ref{fig:associahedraCeballosSantosZiegler} left). Corollary~\ref{cor:associahedronByTruncations} shows that the realization by J.-L.~Loday can be obtained that way in all dimensions. Indeed while following the sequence~$\word=\word_1,\dots,\word_\ell$ to fatten the triangle, we can apply the generic transformations described after Lemma~\ref{lem:realizability} to realize the successive suspensions and stellar subdivisions of edges of the current subword complex. We first associate to each letter of the word~$\word$ the zero vector~$\mathbf{0}$ in~$\Rmbb^0$, and we take the convention that when applying a reverse~$0$-Hecke move to a letter~$\sq{q}=\sq{s}_1$ labeled~$(i,1)$ in the word~$\word_r$, and associated to a ray~$\ray[]$ in the current fan realizing~$\subwordComplex(\word_r)$ (say of dimension~$d\in\Nmbb$), the resulting letters labeled~$(i,1)$ and~$(i,1)'$ are respectively associated to the vectors~$\ray[]\oplus(-\f)$ and~$\ray[]\oplus\f$ in the new fan realizing the suspension~$\subwordComplex(\word_{r+1})\cong\subwordComplex(\word_r)\join\{u_1,u_2\}$, where~$\Rmbb^{d+1}=\Rmbb^d\oplus\Rmbb\,\f$. After fattening the triangle, we therefore obtain the pattern of coordinates of Figure~\ref{fig:geometricPatternAssociahedron} which provides rays supporting a complete simplicial fan realizing the associahedron. The reader can refer to~\cite{PilaudSantos-brickPolytope} to check that this fan is isomorphic to the normal fan of the realization of the associahedron as a convex polytope by J.-L~Loday~\cite{Loday}.

\begin{figure}
\centerline{\includegraphics[width=1.3\textwidth]{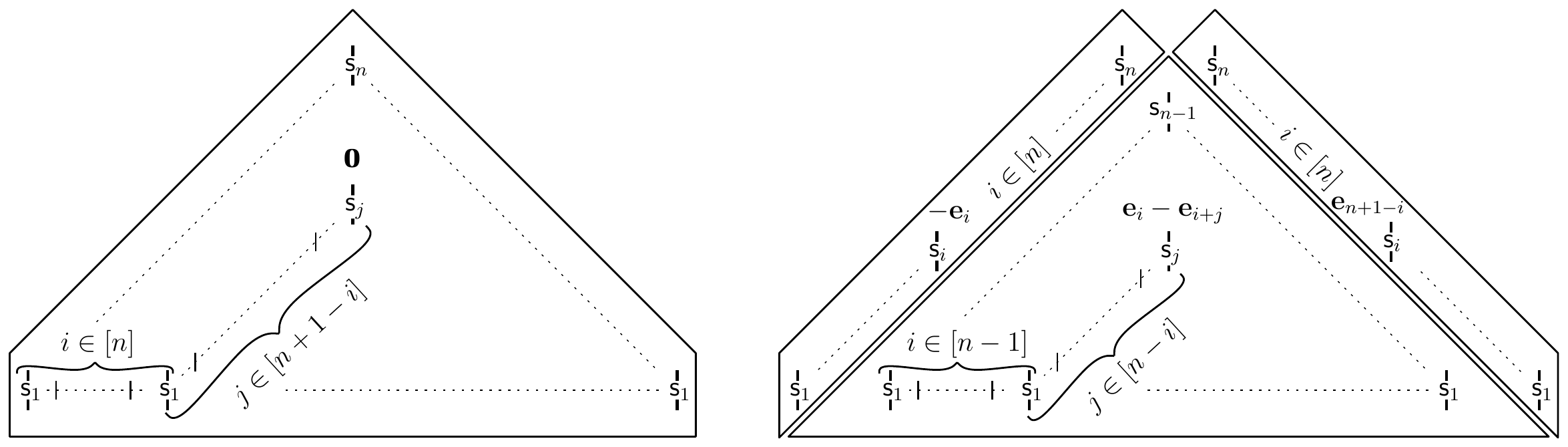}}
\caption{The pattern of coordinates obtained for the associahedron (right) after fattening a triangle (left). The obtained rays are these of the realization of the associahedron by J.-L. Loday~\cite{Loday} (see Figures~\ref{fig:associahedraCeballosSantosZiegler} left and~\ref{fig:associahedraLodayHohlwegLange} left).}
\label{fig:geometricPatternAssociahedron}
\end{figure}

\section{The construction continued to~$2$-associahedra}
\label{sec:2associahedra}

\subsection{Heuristic construction}
\label{subsec:heuristicConstruction}

Cases~(1),~(2) and~(4) of Theorem~\ref{thm:effectBraidMove} are always avoided by the braid moves of a fattening sequence by Lemma~\ref{lem:manyRefinementsLemma}. But we need a geometric transformation implementing the topological effect induced by Case~(5), similar to these after Lemma~\ref{lem:realizability}. Consider a braid move changing a factor~$\sq{q}_r\sq{q}_{r+1}\sq{q}_{r+2}$ of a word~$\word$ by a factor~$\sq{q}_r'\sq{q}_{r+1}'\sq{q}_{r+2}'$ of a word~$\word'$. Suppose that the subword complex~$\subwordComplex(\word)$ is realized by a fan~$\fan_\word$ in which the vertices~$\sq{q}_r,\sq{q}_{r+1}$ and~$\sq{q}_{r+2}$ of~$\subwordComplex(\word)$ are associated to rays generated by some vectors~$\ray[r],\ray[r+1]$ and~$\ray[r+2]$ respectively, that we identify to the rays themselves. Recall that we described the topological effect in Case~(5) of Theorem~\ref{thm:effectBraidMove} as ``moving the vertex~$\sq{q}_{r+1}$ from one side of the edge~$\sq{q}_r\sq{q}_{r+2}$ to the other'' (see Figure~\ref{fig:topologicalBraidMove}). In fan words, it heuristically means that the cone~$\cone\{\ray[r],\ray[r+2]\}$ should separate the ray~$\ray[r+1]$, associated to~$\sq{q}_{r+1}$ in the fan~$\fan_\word$, and the ray~$\ray[r+1]'$ associated to~$\sq{q}_{r+1}'$ in a potential fan realizing the subword complex~$\subwordComplex(\word')$. This intuitive description can be geometrically translated~as~follows.

\begin{compactitem}
 \item take any vector~$\ray[]$ in the interior of the cone~$\cone\{\ray[r],\ray[r+2]\}$, that is~$\ray[]$ can be written in the form~$\alpha\,\ray[r]+\beta\,\ray[r+2]$ for some~$\alpha,\beta>0$;
 \item move~$\ray[r+1]$ in the direction of~$\ray[]$ in order to cross the cone~$\cone\{\ray[r],\ray[r+2]\}$.
\end{compactitem}
For the last point, the ray~$\ray[r+1]$ should be moved ``not to far'' from~$\ray[]$ in order to ensure it to cross the cone~$\cone\{\ray[r],\ray[r+2]\}$, but no other cone of the fan. Since the direction from~$\ray[r+1]$ to~$\ray[]$ is~$\ray[]-\ray[r+1]$, our intuitive description suggests to replace the ray~$\ray[r+1]$ by the ray~$\ray[r+1]'\eqdef\ray[]+\varepsilon\,(\ray[]-\ray[r+1])$, with~$\varepsilon>0$ small enough. Since we are interested in rays, we can consider their generators up to rescaling and therefore say that the ray generated by the vector~$\ray[r+1]$ could be replaced by the ray generated by the vector~$\ray[r+1]'\eqdef\ray[]-\varepsilon\,\ray[r+1]$. In other words, any vector of the form~$\alpha\,\ray[r]+\beta\,\ray[r+2]-\varepsilon\,\ray[r+1]$ with~$\alpha,\beta,\varepsilon>0$ and~$\varepsilon$ small enough would be a legitimate candidate for~$\ray[r+1]'$. We generically choose~$\alpha=1,\beta=1$ and~$\varepsilon=1$ (see~Figure~\ref{fig:braidMoveFan} for an illustration). Finally, following the identifications between the letters of~$\word$ and~$\word'$, the letter~$\sq{q}_r'$ (resp.~$\sq{q}_{r+2}'$) is associated to the ray~$\ray[r+2]$ (resp.~$\ray[r]$).

\begin{figure}
\centerline{\includegraphics[width=1.12\textwidth]{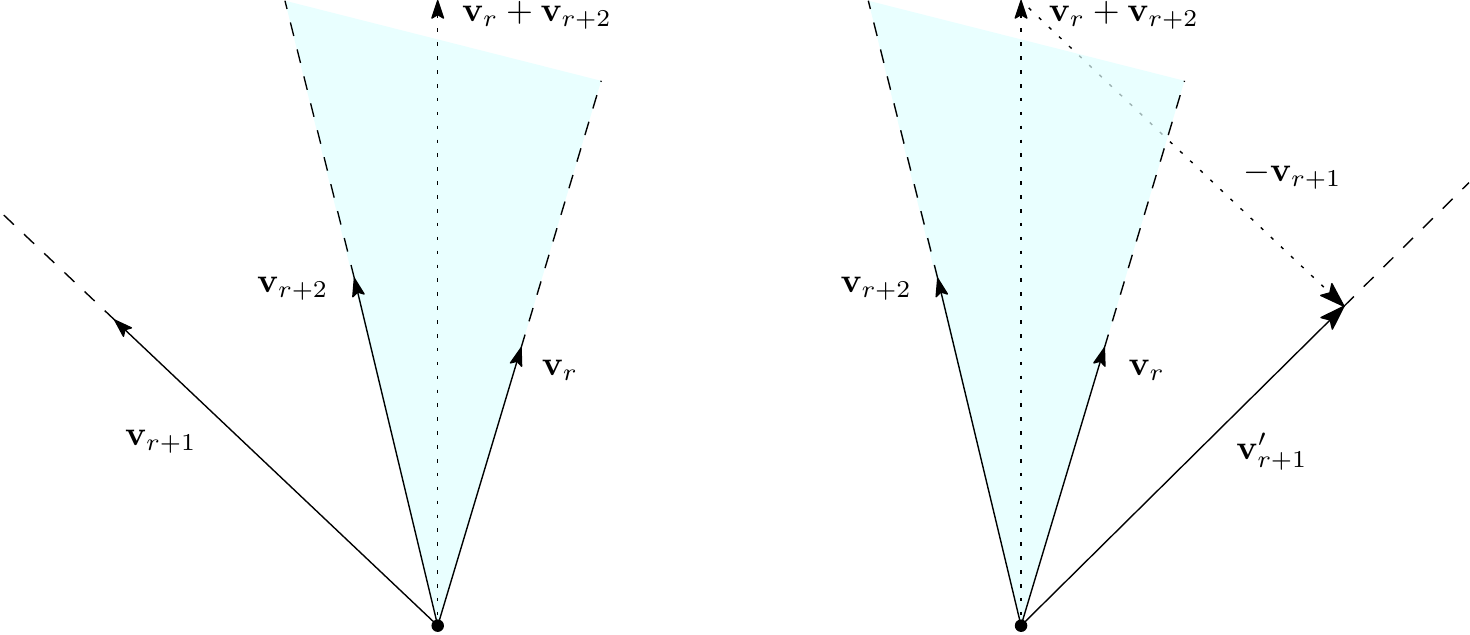}}
\caption{The geometric transformation on the rays of a complete simplicial fan corresponding to a braid move in Case~(5) of Theorem~\ref{thm:effectBraidMove}. In this figure only the relevant rays and cones are represented. The dotted vectors are represented in order to help understanding the figure but they are not rays of the fan.}
\label{fig:braidMoveFan}
\end{figure}

In the sequel, we will consider vectors~$\ray[\sq{q}]\in\Rmbb^d$ associated to the letters~$\sq{q}$ of some words~$\word$ (with~$d=\dim(\subwordComplex(\word))+1$), that we will then abusively call the~\defn{rays} of the subword complex~$\subwordComplex(\word)$, even if the rays in~$\{\ray[\sq{q}]\,|\,\sq{q}\text{ is a letter in }\word\}$ may not support a complete simplicial fan realizing~$\subwordComplex(\word)$. If they do, we say that these rays are~\defn{realizing} for~$\subwordComplex(\word)$. We only require that non-vertices are associated to the zero vector~$\mathbf{0}$. The previous description allows to derive a general heuristic formula for the rays obtained after a fattening sequence. Notice that we already gave after Lemma~\ref{lem:realizability} some transformations on rays associated to one-point-suspensions and stellar subdivisions, that correspond to the effect of moves described in Theorem~\ref{thm:effectNilHeckeMove} and Case~(3) of Theorem~\ref{thm:effectBraidMove}. We observe that the transformation that we defined for Case~(5) of Theorem~\ref{thm:effectBraidMove} is in fact also valid for Case~(3) of Theorem~\ref{thm:effectBraidMove}, since we impose that non-vertices are associated to the zero vector. So given a word~$\word=\factor_1\sq{w}_\circ(\sq{c})\factor_2$ containing a distinguished triangle, and to the letters of which rays are associated, we can use Lemma~\ref{lem:manyRefinementsLemma} and the transformations corresponding to moves to derive rays associated to the letters of the word~$\word'=\factor_1\sq{w}_\circ(\sq{c})\sq{c}^{-1}\factor_2$ obtained by fattening the distinguished triangle of~$\word$. The rays corresponding to the letters of the word~$\word$ span a vector space isomorphic to~$\Rmbb^d$ (with~$d=\dim(\subwordComplex(\word))+1$), and we consider a basis~$\f_1,\dots,\f_n$ of~$\Rmbb^n$ in direct sum with this vector space, so that each reverse~$0$-Hecke move of a fattening sequence lets a coordinate corresponding to one of the~$\f_i$ appear in the rays of the current subword complex. The resulting pattern for the rays of the letters in the factor~$\sq{w}_\circ(\sq{c})\sq{c}^{-1}$ of the word~$\word'$ is presented in Figure~\ref{fig:geometricPatternGeneral}.

\begin{figure}
\centerline{\includegraphics[width=1.3\textwidth]{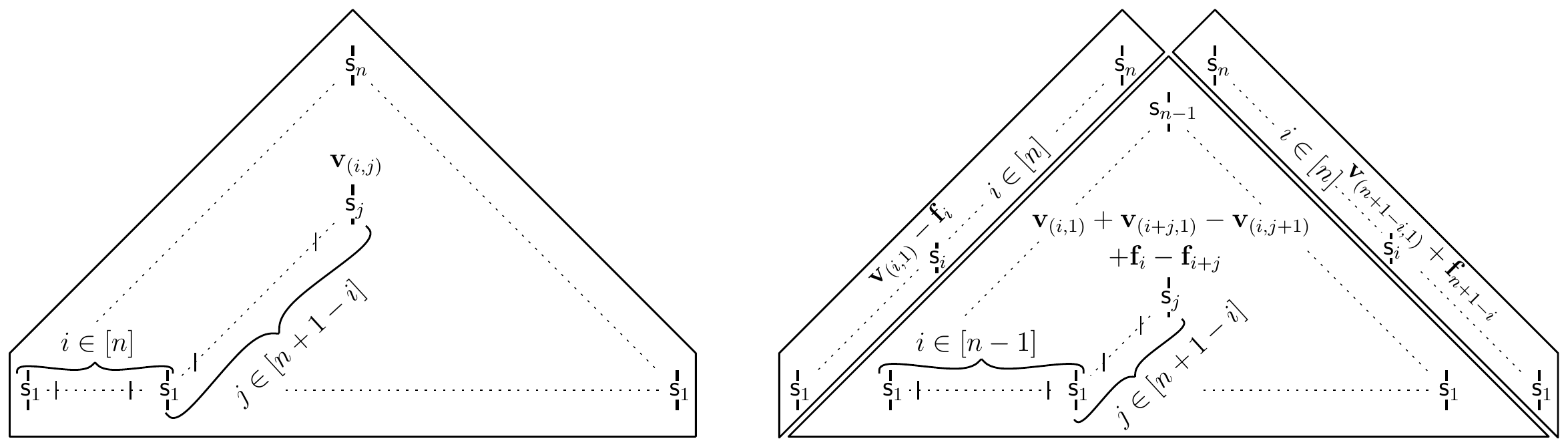}}
\caption{The sorting network of a triangle in a word~$\word$, in which the letter in position~$(i,j)$ is labeled with its associated ray~$\ray[(i,j)]\in\Rmbb^d$, for~$d=\dim(\subwordComplex(\word))-1$ (left) and the sorting network of the factor~$\sq{w}_\circ(\sq{c})\sq{c}^{-1}$ in the word~$\word'$ obtained from~$\word$ by fattening the triangle (right). The letters in this factor are again labeled with their associated ray, obtained from these in the initial triangle following the geometric transformations along the fattening sequence.}
\label{fig:geometricPatternGeneral}
\end{figure}

The pattern of Figure~\ref{fig:geometricPatternGeneral} gives an algorithmic way to produce candidates rays for a fan realization of the subword complex~$\subwordComplex(\word')$ whenever we already know that the set of rays we started with for the subword complex~$\subwordComplex(\word)$ support a complete simplicial fan realizing it. For~$2$-associahedra, we only need to fatten twice a triangle. Indeed, the word~$\sq{w}_\circ(\sq{c})$ can be fattened into the word~$\sq{w}_\circ(\sq{c})\sq{c}^{-1}$, which can be moved by commutations moves to the word~$\sq{c}\sq{w}_\circ(\sq{c})$. In this last word, we can fatten the suffix triangle into obtaining the word~$\sq{c}\sq{w}_\circ(\sq{c})\sq{c}^{-1}$, that we can move once more to~$\sq{c}^2\sq{w}_\circ(\sq{c})$ by commutations moves. The resulting rays for the subword complex~$\subwordComplex(\sq{c}^2\sq{w}_\circ(\sq{c}))$ are illustrated in Table~\ref{tab:2associahedron4Natural} for~$n=4$.

\begin{table}
\centerline{\resizebox{\textwidth}{!}{\tiny \begin{tabular}{@{}>{$}l<{\,$: $[$}@{}>{$}c<{$}*{6}{>{$}c<{$}}>{$}c<{$}@{}>{$]$}c}
\sq{s}_1 & -1\,\,\,\,\,\, & \,\,\,\,\,\,0\,\,\,\,\,\, & \,\,\,\,\,\,0\,\,\,\,\,\, & \,\,\,\,\,\,0\,\,\,\,\,\, & \,\,\,\,\,\,0\,\,\,\,\,\, & \,\,\,\,\,\,0\,\,\,\,\,\, & \,\,\,\,\,\,0\,\,\,\,\,\, & \,\,\,\,\,\,0 & \\
\sq{s}_2 & 0\,\,\,\,\,\, & \,\,\,\,\,\,-1\,\,\,\,\,\, & \,\,\,\,\,\,0\,\,\,\,\,\, & \,\,\,\,\,\,0\,\,\,\,\,\, & \,\,\,\,\,\,0\,\,\,\,\,\, & \,\,\,\,\,\,0\,\,\,\,\,\, & \,\,\,\,\,\,0\,\,\,\,\,\, & \,\,\,\,\,\,0 & \\
\sq{s}_3 & 0\,\,\,\,\,\, & \,\,\,\,\,\,0\,\,\,\,\,\, & \,\,\,\,\,\,-1\,\,\,\,\,\, & \,\,\,\,\,\,0\,\,\,\,\,\, & \,\,\,\,\,\,0\,\,\,\,\,\, & \,\,\,\,\,\,0\,\,\,\,\,\, & \,\,\,\,\,\,0\,\,\,\,\,\, & \,\,\,\,\,\,0 & \\
\sq{s}_4 & 0\,\,\,\,\,\, & \,\,\,\,\,\,0\,\,\,\,\,\, & \,\,\,\,\,\,0\,\,\,\,\,\, & \,\,\,\,\,\,-1\,\,\,\,\,\, & \,\,\,\,\,\,0\,\,\,\,\,\, & \,\,\,\,\,\,0\,\,\,\,\,\, & \,\,\,\,\,\,0\,\,\,\,\,\, & \,\,\,\,\,\,0 & \\
\sq{s}_1 & 1\,\,\,\,\,\, & \,\,\,\,\,\,-1\,\,\,\,\,\, & \,\,\,\,\,\,0\,\,\,\,\,\, & \,\,\,\,\,\,0\,\,\,\,\,\, & \,\,\,\,\,\,-1\,\,\,\,\,\, & \,\,\,\,\,\,0\,\,\,\,\,\, & \,\,\,\,\,\,0\,\,\,\,\,\, & \,\,\,\,\,\,0 & \\
\sq{s}_2 & 0\,\,\,\,\,\, & \,\,\,\,\,\,1\,\,\,\,\,\, & \,\,\,\,\,\,-1\,\,\,\,\,\, & \,\,\,\,\,\,0\,\,\,\,\,\, & \,\,\,\,\,\,0\,\,\,\,\,\, & \,\,\,\,\,\,-1\,\,\,\,\,\, & \,\,\,\,\,\,0\,\,\,\,\,\, & \,\,\,\,\,\,0 &\\
\sq{s}_3 & 0\,\,\,\,\,\, & \,\,\,\,\,\,0\,\,\,\,\,\, & \,\,\,\,\,\,1\,\,\,\,\,\, & \,\,\,\,\,\,-1\,\,\,\,\,\, & \,\,\,\,\,\,0\,\,\,\,\,\, & \,\,\,\,\,\,0\,\,\,\,\,\, & \,\,\,\,\,\,-1\,\,\,\,\,\, & \,\,\,\,\,\,0 &\\
\sq{s}_4 & 0\,\,\,\,\,\, & \,\,\,\,\,\,0\,\,\,\,\,\, & \,\,\,\,\,\,0\,\,\,\,\,\, & \,\,\,\,\,\,1\,\,\,\,\,\, & \,\,\,\,\,\,0\,\,\,\,\,\, & \,\,\,\,\,\,0\,\,\,\,\,\, & \,\,\,\,\,\,0\,\,\,\,\,\, & \,\,\,\,\,\,-1 &\\
\sq{s}_1 & 0\,\,\,\,\,\, & \,\,\,\,\,\,0\,\,\,\,\,\, & \,\,\,\,\,\,0\,\,\,\,\,\, & \,\,\,\,\,\,0\,\,\,\,\,\, & \,\,\,\,\,\,1\,\,\,\,\,\, & \,\,\,\,\,\,-1\,\,\,\,\,\, & \,\,\,\,\,\,0\,\,\,\,\,\, & \,\,\,\,\,\,0 &\\
\sq{s}_2 & 0\,\,\,\,\,\, & \,\,\,\,\,\,-1\,\,\,\,\,\, & \,\,\,\,\,\,1\,\,\,\,\,\, & \,\,\,\,\,\,0\,\,\,\,\,\, & \,\,\,\,\,\,1\,\,\,\,\,\, & \,\,\,\,\,\,0\,\,\,\,\,\, & \,\,\,\,\,\,-1\,\,\,\,\,\, & \,\,\,\,\,\,0 &\\
\sq{s}_3 & 0\,\,\,\,\,\, & \,\,\,\,\,\,-1\,\,\,\,\,\, & \,\,\,\,\,\,0\,\,\,\,\,\, & \,\,\,\,\,\,1\,\,\,\,\,\, & \,\,\,\,\,\,1\,\,\,\,\,\, & \,\,\,\,\,\,0\,\,\,\,\,\, & \,\,\,\,\,\,0\,\,\,\,\,\, & \,\,\,\,\,\,-1 &\\
\sq{s}_4 & 1\,\,\,\,\,\, & \,\,\,\,\,\,-1\,\,\,\,\,\, & \,\,\,\,\,\,0\,\,\,\,\,\, & \,\,\,\,\,\,0\,\,\,\,\,\, & \,\,\,\,\,\,1\,\,\,\,\,\, & \,\,\,\,\,\,0\,\,\,\,\,\, & \,\,\,\,\,\,0\,\,\,\,\,\, & \,\,\,\,\,\,0 &\\
\sq{s}_1 & 0\,\,\,\,\,\, & \,\,\,\,\,\,0\,\,\,\,\,\, & \,\,\,\,\,\,0\,\,\,\,\,\, & \,\,\,\,\,\,0\,\,\,\,\,\, & \,\,\,\,\,\,0\,\,\,\,\,\, & \,\,\,\,\,\,1\,\,\,\,\,\, & \,\,\,\,\,\,-1\,\,\,\,\,\, & \,\,\,\,\,\,0 &\\
\sq{s}_2 & 0\,\,\,\,\,\, & \,\,\,\,\,\,0\,\,\,\,\,\, & \,\,\,\,\,\,-1\,\,\,\,\,\, & \,\,\,\,\,\,1\,\,\,\,\,\, & \,\,\,\,\,\,0\,\,\,\,\,\, & \,\,\,\,\,\,1\,\,\,\,\,\, & \,\,\,\,\,\,0\,\,\,\,\,\, & \,\,\,\,\,\,-1& \\
\sq{s}_3 & 0\,\,\,\,\,\, & \,\,\,\,\,\,1\,\,\,\,\,\, & \,\,\,\,\,\,-1\,\,\,\,\,\, & \,\,\,\,\,\,0\,\,\,\,\,\, & \,\,\,\,\,\,0\,\,\,\,\,\, & \,\,\,\,\,\,1\,\,\,\,\,\, & \,\,\,\,\,\,0\,\,\,\,\,\, & \,\,\,\,\,\,0 &\\
\sq{s}_1 & 0\,\,\,\,\,\, & \,\,\,\,\,\,0\,\,\,\,\,\, & \,\,\,\,\,\,0\,\,\,\,\,\, & \,\,\,\,\,\,0\,\,\,\,\,\, & \,\,\,\,\,\,0\,\,\,\,\,\, & \,\,\,\,\,\,0\,\,\,\,\,\, & \,\,\,\,\,\,1\,\,\,\,\,\, & \,\,\,\,\,\,-1 &\\
\sq{s}_2 & 0\,\,\,\,\,\, & \,\,\,\,\,\,0\,\,\,\,\,\, & \,\,\,\,\,\,1\,\,\,\,\,\, & \,\,\,\,\,\,-1\,\,\,\,\,\, & \,\,\,\,\,\,0\,\,\,\,\,\, & \,\,\,\,\,\,0\,\,\,\,\,\, & \,\,\,\,\,\,1\,\,\,\,\,\, & \,\,\,\,\,\,0 &\\
\sq{s}_1 & 0\,\,\,\,\,\, & \,\,\,\,\,\,0\,\,\,\,\,\, & \,\,\,\,\,\,0\,\,\,\,\,\, & \,\,\,\,\,\,1\,\,\,\,\,\, & \,\,\,\,\,\,0\,\,\,\,\,\, & \,\,\,\,\,\,0\,\,\,\,\,\, & \,\,\,\,\,\,0\,\,\,\,\,\, & \,\,\,\,\,\,1 &\\
 \end{tabular}}}
 \vspace{.4cm}
\caption{The coordinates of the rays associated to the letters of the word~$\sq{c}^{2}\sq{w}_\circ(\sq{c})$ obtained by fattening twice a triangle for~$n=4$. These rays are the first candidates that we obtain to support a complete simplicial fan realizing the~$2$-associahedron~$\multiassociahedron_{2,4}$. However they do not support such fan.}
\vspace{-.5cm}
\label{tab:2associahedron4Natural}
\end{table}

We now wonder whether the rays we obtained support a complete simplicial fan realizing the subword complexes~$\subwordComplex(\sq{c}^2\sq{w}_\circ(\sq{c}))$, and if not ``how close'' it is to be so. For this we consider the set~$\fan_n$ of all cones generated by any set of rays which corresponds to a face of the subword complex~$\subwordComplex(\sq{c}^2\sq{w}_\circ(\sq{c}))$. We recall that a ridge of the subword complex~$\subwordComplex(\sq{c}^2\sq{w}_\circ(\sq{c}))$ is a face which is the intersection of exactly two facets. We will abusively refer to the cones of~$\fan_n$ corresponding to facets (resp. ridges) of the subword complex~$\subwordComplex(\sq{c}^2\sq{w}_\circ(\sq{c}))$ as to the~\emph{facets} (resp.~\emph{ridges}) of~$\fan_n$. The rays of~$\fan_n$ lie in~$\Rmbb^{2n}$ and any ridge is contained in two facets, each generated by~$2n$ vectors, that differ by a single generator. Therefore a ridge~$\ridge$ defines exactly~$2n+1$ rays, and thus at least one linear dependence between them. If the rays associated to the ridge~$\ridge$ are link by a single (up to rescaling) linear dependence not satisfying Condition~(2) of Proposition~\ref{prop:characterizationSimplicialFans}, we say that~$\ridge$ is a~\defn{bad ridge} of~$\fan_n$. If the space of linear dependences on the rays defined by~$\ridge$ has dimension greater than~$1$, we say that~$\ridge$ is a~\defn{degenerate ridge} of~$\fan_n$. In this case at least one of the facets containing the ridge~$\ridge$ is not full dimensional. We call such a facet a~\defn{degenerate cone} of~$\fan_n$. Proposition~\ref{prop:characterizationSimplicialFans} suggests to look at the following statistics on the set of cones~$\fan_n$.

\begin{compactitem}
 \item The rate of bad ridges in~$\fan_n$, which sort of measures ``non tractable issues''.
 \item The rates of degenerate cones and ridges in~$\fan_n$, which describe the ``global degeneracy'' in~$\fan_n$. Since the dual graph of the complex~$\subwordComplex(\sq{c}^2\sq{w}_\circ(\sq{c}))$ is regular, they also give the number of pairs of adjacent degenerate cones.
 \item The minimal dimension of a facet in~$\fan_n$, which measures ``local degeneracy''.
\end{compactitem}

We gather these statistics in Table~\ref{tab:degeneraciesNaive} for~$n\le8$.

\begin{table}
\vspace{.3cm}
\centerline{\small\begin{tabular}{@{}K{.2\textwidth}*{8}{K{.11\textwidth}}}
\toprule
n &1 & 2 & 3 & 4 & 5 & 6 & 7 & 8\\
\text{dimension of }\multiassociahedron_{2,n} & 1 & 3 & 5 & 7 & 9 & 11 & 13 & 15 \\
\hline
\# \text{\bf bad ridges } & \textbf{0} & \textbf{0} & \textbf{0} & \textbf{0} & \textbf{0} & \textbf{0} & \textbf{0} & \textbf{0} \\
\hline
\#\text{ degenerate ridges } & 0 & 0 & 11 & 282 & 5,058 & 78,904 & 1,144,499 & 15,909,182 \\
\#\text{ ridges }                   & 3 & 28 & 252 & 2,376 & 23,595 & 245,388 & 2,654,652 & 29,695,328 \\
\text{ratio (\%)}                  & 0 & 0 & 4.37 & 11.87 & 21.44 & 32.15  & 43.11 & 53.57 \\
\hline
\#\text{ degenerate cones } & 0 & 0 & 2 & 48 & 782 & 10,992 & 143,838 & 1,811,972 \\
\#\text{ cones }                   & 3 & 14 & 84 & 594 & 4,719 & 40,898 & 379,236 & 3,711,916 \\
\text{ratio (\%)}                  & 0 & 0 & 2.38 & 8.08 & 16.57 & 26.88  & 37.93 & 48.82 \\
\hline
\text{minimal dimension } & 2 & 4 & 5 & 6 & 7 & 8 & 9 & 10 \\
\bottomrule
\end{tabular}}
\vspace{.2cm}
\caption{The statistics for the sets of cones~$\fan_n$.}
\vspace{-.3cm}
\label{tab:degeneraciesNaive}
\end{table}

\begin{observation}
\label{ob:naive}
The empirical data suggest that fattening twice a triangle produce rays that do not realize the~$2$-associahedron, but ``only'' up to degeneracies. Indeed the process does not seem to let bad ridges appear. Yet the indicators for degeneracy are high, so that the rays we obtain should not be perturbed easily into realizing ones.
\end{observation}

\subsection{Degrees of freedom}

In view of Observation~\ref{ob:naive}, we need a less naive construction to obtain realizing rays for~$2$-associahedra. We derive it from making the one presented in Section~\ref{subsec:heuristicConstruction} less generic. Indeed we always used the coefficients that we called generic in the geometric translations for the different topological effects of the braid moves. But as we notice after Lemma~\ref{lem:realizability} and at the beginning of Section~\ref{subsec:heuristicConstruction}, we may let some of them vary for the new rays of the letters implied in a reverse~$0$-Hecke move or in a braid move. This belongs to the following many degrees of freedom that we could consider for the construction.
\begin{description}
 \item[coefficients for reverse~$0$-Hecke moves] For any reverse~$0$-Hecke move, we can turn the ray~$\ray[]$ of the doubled letter into~$\ray[]\oplus\alpha\,\f$ and~$\ray[]\oplus\beta\,\f$, for any~$\alpha$ and~$\beta$ satisfying~$\alpha\beta<0$, to realize the corresponding one-point-suspension. The construction of Section~\ref{subsec:heuristicConstruction} generically keeps~$\alpha=-1$ and~$\beta=1$.
  \item[coefficients for braid moves] According to Lemma~\ref{lem:manyRefinementsLemma}, a braid move in a fattening sequence always implies a factor~$\sq{q}_r\sq{q}_{r+1}\sq{q}_{r+2}$ with letters~$\sq{q}_r,\sq{q}_{r+1}$ and~$\sq{q}_{r+2}$ respectively labeled~$(i,1)',(i,j+1)$ and~$(i+j,1)$ (for~$i\in[n-1],j\in[n-i]$). We denote the respective rays associated to these letters~$\ray[r],\ray[r+1]$ and~$\ray[r+2]$, and by~$\sq{q}_r'\sq{q}_{r+1}'\sq{q}_{r+2}'$ the factor by which~$\sq{q}_r\sq{q}_{r+1}\sq{q}_{r+2}$ is replaced by the braid move. We chose in Section~\ref{subsec:heuristicConstruction} to associate the letter~$\sq{q}_{r+1}'$ to any ray of the form~$\alpha\,\ray[r]+\beta\,\ray[r+2] -\varepsilon\,\ray[r+1]$ with~$\alpha>0,\beta>0,\varepsilon>0$ and~$\varepsilon$ small enough. In our construction of Section~\ref{subsec:heuristicConstruction}, we fatten twice a triangle. The first fattening sequence only contains braid move inducing stellar subdivisions while the second one only contains braid move inducing Case~(5) of Theorem~\ref{thm:effectBraidMove}. In the first case there is in fact only two choices of coefficients since the ray associated to the letter~$\sq{q}_{r+1}$ is the zero vector. We will denote by~$\coefLeft$ and~$\coefRight$ the respective coefficients of~$\ray[r]$ and~$\ray[r+2]$ in the first fattening sequence, and by~$\coefA,\coefB$ and~$\coefC$ the respective coefficients of~$\ray[r],\ray[r+1]$ and~$\ray[r+2]$ in the second fattening sequence of the construction. Since the effect of braid moves are local, we can~\apriori choose all these coefficients independently whereas the initial construction of Section~\ref{subsec:heuristicConstruction} generically set~$\coefLeft=\coefRight=\coefA=\coefB=\coefC=1$.
 \item[choice of the triangle] We did not insist on the triangle that we fatten in the construction. There is indeed only one choice for the initial word, which is itself a triangle, but the second fattening sequence is applied to the suffix triangle of the word~$\sq{c}\sq{w}_\circ(\sq{c})$. This word can be moved to the word~$\sq{w}_\circ(\sq{c})\sq{c}^{-1}$ by commutation moves so that we could also apply the second fattening sequence to the prefix triangle of this new word. It is easy to check that the two sets of rays obtained by both methods are linearly equivalent. Yet we can use the rotation map described in Theorem~\ref{thm:rotationMap} to obtain other non equivalent constructions. Denoting by~$\word^{\rotated k}$ the word obtained by applying~$k$ times the rotation map to a word~$\word$, we see that any word of the form~$\left(\sq{c}\sq{w}_\circ(\sq{c})\right)^{\rotated k.n}$, for~$k\in\Nmbb$, can be moved back to~$\sq{c}\sq{w}_\circ(\sq{c})$ by commutation moves. Therefore we could choose a lot of triangles to fatten (in fact~$n+1$) instead of always taking the suffix one without applying any rotation to the current word, as in the construction of Section~\ref{subsec:heuristicConstruction}.
 \item[starting associahedron] Finally we observe that we could start from any realizing rays for the subword complex~$\subwordComplex(\sq{c}\sq{w}_\circ(\sq{c}))$ to apply the second fattening sequence of the construction. In view of the numerous fan realizations for the usual associahedron, this is a wide additional degree of freedom.
\end{description}

We did not test exhaustively all the possibilities allowed by these multiple degrees of freedom. Since the initial motivation of this project was to realize as fans one of the first unrealized multiassociahedra~$\multiassociahedron_{2,5}$ and~$\multiassociahedron_{4,4}$, we mostly made some kind of ``depth first search testing'' in that direction. Therefore we will not mention all combinations that failed out and concentrate on this that actually provided results. It turns out that letting the coefficients~$\coefLeft$ and~$\coefRight$ vary was somehow successful. So from now on we will denote by~$\fan_n(\coefLeft,\coefRight)$ the set of cones obtained by fattening twice a suffix triangle of an initial triangle, where the first fattening is done with coefficients~$\coefLeft$ and~$\coefRight$, and the second one with coefficients~$\coefA=\coefB=\coefC=1$. The choice~$\coefLeft=5$ and~$\coefRight=3$ was the best one among these not letting the coefficients depend on the position of the corresponding letter. Table~\ref{tab:degeneraciesFixedLeftRight} gathers the statistics for the set of cones~$\fan_n(5,3)$ for~$n\le8$. Observe that the rates of degenerate ridges and cones decreases by a factor of about~$2$ with this simple change in the coefficients. In particular we obtain new realizing rays for the~$2$-associahedron~$\multiassociahedron_{2,3}$ (see~Table~\ref{tab:validNaiveCoordinate2associahedron4}). We came out with such coefficients mostly because we observed that having~$\coefLeft$ and~$\coefRight$ relatively prime helped reducing degeneracies. So we tried to keep them relatively prime while letting them depend on the position. It turns out that letting them be linear in~$(i,j)$ yielded us the best results, namely for~$\coefLeft=2n+4-i-j$ and~$\coefRight=\coefLeft-1$. We gather in Table~\ref{tab:degeneraciesLinearLeftRight} the statistics of the sets of cones~$\fan_n(2n+4-i-j,2n+3-i-j)$ for~$n\le8$.

\begin{table}
\vspace{.4cm}
\centerline{\small\begin{tabular}{@{}K{.2\textwidth}*{8}{K{.11\textwidth}}}
\toprule
n &1 & 2 & 3 & 4 & 5 & 6 & 7 & 8\\
\text{dimension of }\multiassociahedron_{2,n} & 1 & 3 & 5 & 7 & 9 & 11 & 13 & 15 \\
\hline
\# \text{\bf bad ridges } & \textbf{0} & \textbf{0} & \textbf{0} & \textbf{0} & \textbf{0} & \textbf{0} & \textbf{0} & \textbf{0} \\
\hline
\#\text{ degenerate ridges } & 0 & 0   & 0    & 78   & 2,216   & 43,298 & 724,546 & 11,150,457 \\
\#\text{ ridges }                   & 3 & 28 & 252 & 2,376 & 23,595 & 245,388 & 2,654,652 & 29,695,328 \\
\text{ratio (\%)}                  & 0 & 0   & 0 & 3.28 & 9.39 & 17.63 & 27.29 & 37.55 \\
\hline
\#\text{ degenerate cones } & 0 & 0 & 0 & 12 & 320 & 5,742 & 87,714 & 1,233,154 \\
\#\text{ cones }                   & 3 & 14 & 84 & 594 & 4,719 & 40,898 & 379,236 & 3,711,916 \\
\text{ratio (\%)}                  & 0 & 0 & 0 & 2.02 & 6.78 & 14.04  & 23.13 & 33.22 \\
\hline
\text{minimal dimension } & 2 & 4 & 6 & 7 & 8 & 9 & 10 & 11 \\
\bottomrule
\end{tabular}}
\vspace{.1cm}
\caption{The statistics for the sets of cones~$\fan_n(5,3)$.}
\label{tab:degeneraciesFixedLeftRight}
\end{table}

 \begin{table}
 \centerline{\begin{tabular}{@{}>{$}l<{\,$: $[$}@{}>{$}c<{$}*{4}{>{$}c<{$}}>{$}c<{$}@{}>{$]$}c}
\sq{s}_1 & -1\,\,\,\,\,\, & \,\,\,\,\,\,0\,\,\,\,\,\, & \,\,\,\,\,\,0\,\,\,\,\,\, & \,\,\,\,\,\,0\,\,\,\,\,\, & \,\,\,\,\,\,0\,\,\,\,\,\, & \,\,\,\,\,\,0\,\,\,\,\,\, & \\
\sq{s}_2 & 0\,\,\,\,\,\, & \,\,\,\,\,\,-1\,\,\,\,\,\, & \,\,\,\,\,\,0\,\,\,\,\,\, & \,\,\,\,\,\,0\,\,\,\,\,\, & \,\,\,\,\,\,0\,\,\,\,\,\, & \,\,\,\,\,\,0\,\,\,\,\,\, & \\
\sq{s}_3 & 0\,\,\,\,\,\, & \,\,\,\,\,\,0\,\,\,\,\,\, & \,\,\,\,\,\,-1\,\,\,\,\,\, & \,\,\,\,\,\,0\,\,\,\,\,\, & \,\,\,\,\,\,0\,\,\,\,\,\, & \,\,\,\,\,\,0\,\,\,\,\,\, & \\
\sq{s}_1 & 5\,\,\,\,\,\, & \,\,\,\,\,\,-3\,\,\,\,\,\, & \,\,\,\,\,\,0\,\,\,\,\,\, & \,\,\,\,\,\,-1\,\,\,\,\,\, & \,\,\,\,\,\,0\,\,\,\,\,\, & \,\,\,\,\,\,0\,\,\,\,\,\, & \\
\sq{s}_2 & 0\,\,\,\,\,\, & \,\,\,\,\,\,5\,\,\,\,\,\, & \,\,\,\,\,\,-3\,\,\,\,\,\, & \,\,\,\,\,\,0\,\,\,\,\,\, & \,\,\,\,\,\,-1\,\,\,\,\,\, & \,\,\,\,\,\,0\,\,\,\,\,\, & \\
\sq{s}_3 & 0\,\,\,\,\,\, & \,\,\,\,\,\,0\,\,\,\,\,\, & \,\,\,\,\,\,1\,\,\,\,\,\, & \,\,\,\,\,\,0\,\,\,\,\,\, & \,\,\,\,\,\,0\,\,\,\,\,\, & \,\,\,\,\,\,-1\,\,\,\,\,\, & \\
\sq{s}_1 & 0\,\,\,\,\,\, & \,\,\,\,\,\,2\,\,\,\,\,\, & \,\,\,\,\,\,0\,\,\,\,\,\, & \,\,\,\,\,\,1\,\,\,\,\,\, & \,\,\,\,\,\,-1\,\,\,\,\,\, & \,\,\,\,\,\,0\,\,\,\,\,\, & \\
\sq{s}_2 & 4\,\,\,\,\,\, & \,\,\,\,\,\,-3\,\,\,\,\,\, & \,\,\,\,\,\,1\,\,\,\,\,\, & \,\,\,\,\,\,1\,\,\,\,\,\, & \,\,\,\,\,\,0\,\,\,\,\,\, & \,\,\,\,\,\,-1\,\,\,\,\,\, & \\
\sq{s}_3 & 5\,\,\,\,\,\, & \,\,\,\,\,\,-3\,\,\,\,\,\, & \,\,\,\,\,\,0\,\,\,\,\,\, & \,\,\,\,\,\,1\,\,\,\,\,\, & \,\,\,\,\,\,0\,\,\,\,\,\, & \,\,\,\,\,\,0\,\,\,\,\,\, & \\
\sq{s}_1 & 0\,\,\,\,\,\, & \,\,\,\,\,\,4\,\,\,\,\,\, & \,\,\,\,\,\,-2\,\,\,\,\,\, & \,\,\,\,\,\,0\,\,\,\,\,\, & \,\,\,\,\,\,1\,\,\,\,\,\, & \,\,\,\,\,\,-1\,\,\,\,\,\, & \\
\sq{s}_2 & 0\,\,\,\,\,\, & \,\,\,\,\,\,5\,\,\,\,\,\, & \,\,\,\,\,\,-3\,\,\,\,\,\, & \,\,\,\,\,\,0\,\,\,\,\,\, & \,\,\,\,\,\,1\,\,\,\,\,\, & \,\,\,\,\,\,0\,\,\,\,\,\, & \\
\sq{s}_1 & 0\,\,\,\,\,\, & \,\,\,\,\,\,0\,\,\,\,\,\, & \,\,\,\,\,\,1\,\,\,\,\,\, & \,\,\,\,\,\,0\,\,\,\,\,\, & \,\,\,\,\,\,0\,\,\,\,\,\, & \,\,\,\,\,\,1\,\,\,\,\,\, & \\
 \end{tabular}}
 \vspace{.4cm}
 \caption{The rays supporting the set of cones~$\fan_3(5,3)$, associated to each letter of the word~$\sq{c}^2\sq{w}_\circ(\sq{c})$ for~$n=3$. These rays are realizing, that is the set of cones~$\fan_3(5,3)$ is a complete simplicial fan realizing the~$2$-associahedron~$\multiassociahedron_{2,3}$.}
 \label{tab:validNaiveCoordinate2associahedron4}
 \end{table}

\begin{table}
\centerline{\small\begin{tabular}{@{}K{.2\textwidth}*{8}{K{.11\textwidth}}}
\toprule
n &1 & 2 & 3 & 4 & 5 & 6 & 7 & 8\\
\text{dimension of }\multiassociahedron_{2,n} & 1 & 3 & 5 & 7 & 9 & 11 & 13 & 15 \\
\hline
\# \text{\bf bad ridges } & \textbf{0} & \textbf{0} & \textbf{0} & \textbf{0} & \textbf{0} & \textbf{0} & \textbf{0} & \textbf{20} \\
\hline
\#\text{ degenerate ridges } & 0 & 0   & 0    & 39   & 1,122   & 22,317 & 381,533 & 6,026,814 \\
\#\text{ ridges }                   & 3 & 28 & 252 & 2,376 & 23,595 & 245,388 & 2,654,652 & 29,695,328 \\
\text{ratio (\%)}               & 0 & 0   & 0 & 1.64 & 4.76 & 9.09 & 14.37 & 20.30 \\
\hline
\#\text{ degenerate cones } & 0 & 0 & 0 & 6 & 160 & 2,904 & 45,173 & 650,734 \\
\#\text{ cones }                   & 3 & 14 & 84 & 594 & 4,719 & 40,898 & 379,236 & 3,711,916 \\
\text{ratio (\%)}                  & 0 & 0 & 0 & 1.01 & 3.39 & 7.10 & 11.91 & 17.53 \\
\hline
\text{minimal dimension } & 2 & 4 & 6 & 7 & 8 & 9 & 10 & 11 \\
\bottomrule
\end{tabular}}
\vspace{.1cm}
\caption{The statistics for the sets of cones~$\fan_n(2n+4-i-j,2n+3-i-j)$. With this choice of coefficients, some bad ridges appear in the construction for~$n=8$.}
\vspace{-.4cm}
\label{tab:degeneraciesLinearLeftRight}
\end{table} 

\begin{observation}
\label{ob:amelioration}
It is possible to let the coefficients~$\coefLeft$ and~$\coefRight$ vary in order to still obtain sets of cones~$\fan_n(\coefLeft,\coefRight)$ with degeneracies but almost no bad ridges. Moreover some choices let the degeneracy indicator decrease remarkably. Indeed the choice~$\coefLeft=2n+4-i-j$ and~$\coefRight=2n+3-i-j$ again decreases by a factor of about~$2$ these indicators by comparison to the choice~$\coefLeft=5$ and~$\coefRight=3$.
\end{observation}

\subsection{Perturbations}

As a particular case of Observation~\ref{ob:amelioration}, we noticed that the set of cones~$\fan_5(14-i-j,13-i-j)$ seemed really close of realizing the~$2$-associahedron~$\multiassociahedron_{2,5}$. So we stopped our experiments on the coefficients~$\coefLeft$ and~$\coefRight$ and tried to perturb the rays randomly with the hope of killing the last remaining degeneracies. Again there is some freedom in this idea of ``perturbing'' the rays. Indeed we tried to add a random and small enough term to each of their coordinate, unsuccessfully. But then, applying the perturbations to the coefficients~$\coefLeft$ and~$\coefRight$ themselves finally gave us realizing rays for the~$2$-associahedron~$\multiassociahedron_{2,5}$, that are given in Table~\ref{tab:2associahedron5Perturbed}. Of course not all perturbations terms that appear in these rays are necessary. So working on the coordinates in Table~\ref{tab:2associahedron5Perturbed}, we found better ones, given in Table~\ref{tab:2associahedron5IntegerCoordinates}. Notice that these rays now have integer coordinates between~$-10$ and~$10$. Moreover we were able to reduce the number of perturbation terms to~$3$.

\begin{table}
\vspace{.6cm}
\centerline{\resizebox{1.58\textwidth}{!}{\tiny \begin{tabular}{@{}>{$}l<{\,$: $[$}@{}>{$}c<{$}*{8}{>{$}c<{$}}>{$}c<{$}@{}>{$]$}c}
\sq{s}_1 & -1 &  0 &  0 &  0 &  0 &  0 &  0 &  0 &  0 &  0 & \\
\sq{s}_2 & 0 &  -1 &  0 &  0 &  0 &  0 &  0 &  0 &  0 &  0 & \\
\sq{s}_3 & 0 &  0 &  -1 &  0 &  0 &  0 &  0 &  0 &  0 &  0 & \\
\sq{s}_4 & 0 &  0 &  0 &  -1 &  0 &  0 &  0 &  0 &  0 &  0 & \\
\sq{s}_5 & 0 &  0 &  0 &  0 &  -1 &  0 &  0 &  0 &  0 &  0 & \\
\sq{s}_1 & 11.995220082449654 &  -11.002018603888557 &  0 &  0 &  0 &  -1 &  0 &  0 &  0 &  0 & \\
\sq{s}_2 & 0 &  10.998025890846899 &  -10.000386365505443 &  0 &  0 &  0 &  -1 &  0 &  0 &  0 & \\
\sq{s}_3 & 0 &  0 &  9.995777402249201 &  -8.998111068535287 &  0 &  0 &  0 &  -1 &  0 &  0 & \\
\sq{s}_4 & 0 &  0 &  0 &  9.00229715779829 &  -8.001163693705028 &  0 &  0 &  0 &  -1 &  0 & \\
\sq{s}_5 & 0 &  0 &  0 &  0 &  1 &  0 &  0 &  0 &  0 &  -1 & \\
\sq{s}_1 & 0.9957828222479908 &  -0.003992713041657936 &  -0.0014994963267600525 &  0 &  0 &  1 &  -1 &  0 &  0 &  0 & \\
\sq{s}_2 & 1.9999160791264572 &  -11.002018603888557 &  9.995777402249201 &  -0.0030142895687461646 &  0 &  1 &  0 &  -1 &  0 &  0 & \\
\sq{s}_3 & 2.9965399682572844 &  -11.002018603888557 &  0 &  9.00229715779829 &  -0.0040464541747882166 &  1 &  0 &  0 &  -1 &  0 & \\
\sq{s}_4 & 10.995220082449654 &  -11.002018603888557 &  0 &  0 &  1 &  1 &  0 &  0 &  0 &  -1 & \\
\sq{s}_5 & 11.995220082449654 &  -11.002018603888557 &  0 &  0 &  0 &  1 &  0 &  0 &  0 &  0 & \\
\sq{s}_1 & 0 &  0.9936127652497042 &  -0.004608963256242049 &  0.0012209718570073136 &  0 &  0 &  1 &  -1 &  0 &  0 & \\
\sq{s}_2 & 0 &  2.001435184525553 &  -10.000386365505443 &  9.00229715779829 &  -0.0008415764775513424 &  0 &  1 &  0 &  -1 &  0 & \\
\sq{s}_3 & 0 &  9.998025890846899 &  -10.000386365505443 &  0 &  1.0 &  0 &  1 &  0 &  0 &  -1 & \\
\sq{s}_4 & 0 &  10.998025890846899 &  -10.000386365505443 &  0 &  0 &  0 &  1 &  0 &  0 &  0 & \\
\sq{s}_1 & 0 &  0 &  0.9945133005526081 &  0.004186089263003012 &  0.0007304432600054866 &  0 &  0 &  1 &  -1 &  0 & \\
\sq{s}_2 & 0 &  0 &  8.995777402249201 &  -8.998111068535287 &  1 &  0 &  0 &  1 &  0 &  -1 & \\
\sq{s}_3 & 0 &  0 &  9.995777402249201 &  -8.998111068535287 &  0 &  0 &  0 &  1 &  0 &  0 & \\
\sq{s}_1 & 0 &  0 &  0 &  8.00229715779829 &  -7.001163693705028 &  0 &  0 &  0 &  1 &  -1 & \\
\sq{s}_2 & 0 &  0 &  0 &  9.00229715779829 &  -8.001163693705028 &  0 &  0 &  0 &  1 &  0 & \\
\sq{s}_1 & 0 &  0 &  0 &  0 &  1 &  0 &  0 &  0 &  0 &  1 & \\
 \end{tabular}}}
 \vspace{.5cm}
\caption{Realizing rays of the~$2$-associahedron~$\multiassociahedron_{2,5}$, associated to each letter of the word~$\sq{c}^2\sq{w}_\circ(\sq{c})$ for~$n=5$. These rays were obtained by fattening twice the suffix triangle of an initial triangle with coefficients~$\coefLeft=14-i-j+\random^\ell$ and~$\coefRight=13-i-j+\random^r$, where~$\random^\ell$ and~$\random^r$ were uniform independent random variables in~$[-0.001,0.001]$ for~$i\in[4]$ and~$j\in[5-i]$.}
\vspace{-1cm}
\label{tab:2associahedron5Perturbed}
\end{table}

\begin{table}
\centerline{\resizebox{\textwidth}{!}{\tiny \begin{tabular}{@{}>{$}l<{\,$: $[$}@{}>{$}c<{$}*{8}{>{$}c<{$}}>{$}c<{$}@{}>{$]$}c}
\sq{s}_1& -1\,\,\,\,\,\, & \,\,\,\,\,\, 0\,\,\,\,\,\, & \,\,\,\,\,\, 0\,\,\,\,\,\, & \,\,\,\,\,\, 0\,\,\,\,\,\, & \,\,\,\,\,\, 0\,\,\,\,\,\, & \,\,\,\,\,\, 0\,\,\,\,\,\, & \,\,\,\,\,\, 0\,\,\,\,\,\, & \,\,\,\,\,\, 0\,\,\,\,\,\, & \,\,\,\,\,\, 0\,\,\,\,\,\, & \,\,\,\,\,\, 0& \\
\sq{s}_2& 0\,\,\,\,\,\, & \,\,\,\,\,\, -1\,\,\,\,\,\, & \,\,\,\,\,\, 0\,\,\,\,\,\, & \,\,\,\,\,\, 0\,\,\,\,\,\, & \,\,\,\,\,\, 0\,\,\,\,\,\, & \,\,\,\,\,\, 0\,\,\,\,\,\, & \,\,\,\,\,\, 0\,\,\,\,\,\, & \,\,\,\,\,\, 0\,\,\,\,\,\, & \,\,\,\,\,\, 0\,\,\,\,\,\, & \,\,\,\,\,\, 0& \\
\sq{s}_3& 0\,\,\,\,\,\, & \,\,\,\,\,\, 0\,\,\,\,\,\, & \,\,\,\,\,\, -1\,\,\,\,\,\, & \,\,\,\,\,\, 0\,\,\,\,\,\, & \,\,\,\,\,\, 0\,\,\,\,\,\, & \,\,\,\,\,\, 0\,\,\,\,\,\, & \,\,\,\,\,\, 0\,\,\,\,\,\, & \,\,\,\,\,\, 0\,\,\,\,\,\, & \,\,\,\,\,\, 0\,\,\,\,\,\, & \,\,\,\,\,\, 0& \\
\sq{s}_4& 0\,\,\,\,\,\, & \,\,\,\,\,\, 0\,\,\,\,\,\, & \,\,\,\,\,\, 0\,\,\,\,\,\, & \,\,\,\,\,\, -1\,\,\,\,\,\, & \,\,\,\,\,\, 0\,\,\,\,\,\, & \,\,\,\,\,\, 0\,\,\,\,\,\, & \,\,\,\,\,\, 0\,\,\,\,\,\, & \,\,\,\,\,\, 0\,\,\,\,\,\, & \,\,\,\,\,\, 0\,\,\,\,\,\, & \,\,\,\,\,\, 0& \\
\sq{s}_5& 0\,\,\,\,\,\, & \,\,\,\,\,\, 0\,\,\,\,\,\, & \,\,\,\,\,\, 0\,\,\,\,\,\, & \,\,\,\,\,\, 0\,\,\,\,\,\, & \,\,\,\,\,\, -1\,\,\,\,\,\, & \,\,\,\,\,\, 0\,\,\,\,\,\, & \,\,\,\,\,\, 0\,\,\,\,\,\, & \,\,\,\,\,\, 0\,\,\,\,\,\, & \,\,\,\,\,\, 0\,\,\,\,\,\, & \,\,\,\,\,\, 0& \\
\sq{s}_1& 12\,\,\,\,\,\, & \,\,\,\,\,\, -11\,\,\,\,\,\, & \,\,\,\,\,\, 0\,\,\,\,\,\, & \,\,\,\,\,\, 0\,\,\,\,\,\, & \,\,\,\,\,\, 0\,\,\,\,\,\, & \,\,\,\,\,\, -1\,\,\,\,\,\, & \,\,\,\,\,\, 0\,\,\,\,\,\, & \,\,\,\,\,\, 0\,\,\,\,\,\, & \,\,\,\,\,\, 0\,\,\,\,\,\, & \,\,\,\,\,\,0& \\
\sq{s}_2& 0\,\,\,\,\,\, & \,\,\,\,\,\, 11\,\,\,\,\,\, & \,\,\,\,\,\, -10\,\,\,\,\,\, & \,\,\,\,\,\, 0\,\,\,\,\,\, & \,\,\,\,\,\, 0\,\,\,\,\,\, & \,\,\,\,\,\, 0\,\,\,\,\,\, & \,\,\,\,\,\, -1\,\,\,\,\,\, & \,\,\,\,\,\, 0\,\,\,\,\,\, & \,\,\,\,\,\, 0\,\,\,\,\,\, & \,\,\,\,\,\, 0& \\
\sq{s}_3& 0\,\,\,\,\,\, & \,\,\,\,\,\, 0\,\,\,\,\,\, & \,\,\,\,\,\, 10\,\,\,\,\,\, & \,\,\,\,\,\, -9\,\,\,\,\,\, & \,\,\,\,\,\, 0\,\,\,\,\,\, & \,\,\,\,\,\, 0\,\,\,\,\,\, & \,\,\,\,\,\, 0\,\,\,\,\,\, & \,\,\,\,\,\, -1\,\,\,\,\,\, & \,\,\,\,\,\, 0\,\,\,\,\,\, & \,\,\,\,\,\, 0& \\
\sq{s}_4& 0\,\,\,\,\,\, & \,\,\,\,\,\, 0\,\,\,\,\,\, & \,\,\,\,\,\, 0\,\,\,\,\,\, & \,\,\,\,\,\, 9\,\,\,\,\,\, & \,\,\,\,\,\, -8\,\,\,\,\,\, & \,\,\,\,\,\, 0\,\,\,\,\,\, & \,\,\,\,\,\, 0\,\,\,\,\,\, & \,\,\,\,\,\, 0\,\,\,\,\,\, & \,\,\,\,\,\, -1\,\,\,\,\,\, & \,\,\,\,\,\, 0& \\
\sq{s}_5& 0\,\,\,\,\,\, & \,\,\,\,\,\, 0\,\,\,\,\,\, & \,\,\,\,\,\, 0\,\,\,\,\,\, & \,\,\,\,\,\, 0\,\,\,\,\,\, & \,\,\,\,\,\, 1\,\,\,\,\,\, & \,\,\,\,\,\, 0\,\,\,\,\,\, & \,\,\,\,\,\, 0\,\,\,\,\,\, & \,\,\,\,\,\, 0\,\,\,\,\,\, & \,\,\,\,\,\, 0\,\,\,\,\,\, & \,\,\,\,\,\, -1& \\
\sq{s}_1& 1\,\,\,\,\,\, & \,\,\,\,\,\, 0\,\,\,\,\,\, & \,\,\,\,\,\, 0\,\,\,\,\,\, & \,\,\,\,\,\, 0\,\,\,\,\,\, & \,\,\,\,\,\, 0\,\,\,\,\,\, & \,\,\,\,\,\, 1\,\,\,\,\,\, & \,\,\,\,\,\, -1\,\,\,\,\,\, & \,\,\,\,\,\, 0\,\,\,\,\,\, & \,\,\,\,\,\, 0\,\,\,\,\,\, & \,\,\,\,\,\, 0& \\
\sq{s}_2& 2\,\,\,\,\,\, & \,\,\,\,\,\, -11\,\,\,\,\,\, & \,\,\,\,\,\, 10\,\,\,\,\,\, & \,\,\,\,\,\, \text{\textcolor{red}{$\boxed{\mathbf{-1}}$}}\,\,\,\,\,\, & \,\,\,\,\,\, 0\,\,\,\,\,\, & \,\,\,\,\,\, 1\,\,\,\,\,\, & \,\,\,\,\,\, 0\,\,\,\,\,\, & \,\,\,\,\,\, -1\,\,\,\,\,\, & \,\,\,\,\,\, 0\,\,\,\,\,\, & \,\,\,\,\,\, 0& \\
\sq{s}_3& 3\,\,\,\,\,\, & \,\,\,\,\,\, -11\,\,\,\,\,\, & \,\,\,\,\,\, 0\,\,\,\,\,\, & \,\,\,\,\,\, 9\,\,\,\,\,\, & \,\,\,\,\,\,\text{\textcolor{red}{$\boxed{\mathbf{-2}}$}} \,\,\,\,\,\, & \,\,\,\,\,\, 1\,\,\,\,\,\, & \,\,\,\,\,\, 0\,\,\,\,\,\, & \,\,\,\,\,\, 0\,\,\,\,\,\, & \,\,\,\,\,\, -1\,\,\,\,\,\, & \,\,\,\,\,\, 0& \\
\sq{s}_4& 11\,\,\,\,\,\, & \,\,\,\,\,\, -11\,\,\,\,\,\, & \,\,\,\,\,\, 0\,\,\,\,\,\, & \,\,\,\,\,\, 0\,\,\,\,\,\, & \,\,\,\,\,\, 1\,\,\,\,\,\, & \,\,\,\,\,\, 1\,\,\,\,\,\, & \,\,\,\,\,\, 0\,\,\,\,\,\, & \,\,\,\,\,\, 0\,\,\,\,\,\, & \,\,\,\,\,\, 0\,\,\,\,\,\, & \,\,\,\,\,\, -1& \\
\sq{s}_5& 12\,\,\,\,\,\, & \,\,\,\,\,\, -11\,\,\,\,\,\, & \,\,\,\,\,\, 0\,\,\,\,\,\, & \,\,\,\,\,\, 0\,\,\,\,\,\, & \,\,\,\,\,\, 0\,\,\,\,\,\, & \,\,\,\,\,\, 1\,\,\,\,\,\, & \,\,\,\,\,\, 0\,\,\,\,\,\, & \,\,\,\,\,\, 0\,\,\,\,\,\, & \,\,\,\,\,\, 0\,\,\,\,\,\, & \,\,\,\,\,\, 0& \\
\sq{s}_1& 0\,\,\,\,\,\, & \,\,\,\,\,\, 1\,\,\,\,\,\, & \,\,\,\,\,\, 0\,\,\,\,\,\, & \,\,\,\,\,\, 0\,\,\,\,\,\, & \,\,\,\,\,\, 0\,\,\,\,\,\, & \,\,\,\,\,\, 0\,\,\,\,\,\, & \,\,\,\,\,\, 1\,\,\,\,\,\, & \,\,\,\,\,\, -1\,\,\,\,\,\, & \,\,\,\,\,\, 0\,\,\,\,\,\, & \,\,\,\,\,\, 0& \\
\sq{s}_2& 0\,\,\,\,\,\, & \,\,\,\,\,\, 2\,\,\,\,\,\, & \,\,\,\,\,\, -10\,\,\,\,\,\, & \,\,\,\,\,\, 9\,\,\,\,\,\, & \,\,\,\,\,\, \text{\textcolor{red}{$\boxed{\mathbf{-1}}$}}\,\,\,\,\,\, & \,\,\,\,\,\, 0\,\,\,\,\,\, & \,\,\,\,\,\, 1\,\,\,\,\,\, & \,\,\,\,\,\, 0\,\,\,\,\,\, & \,\,\,\,\,\, -1\,\,\,\,\,\, & \,\,\,\,\,\, 0& \\
\sq{s}_3& 0\,\,\,\,\,\, & \,\,\,\,\,\, 10\,\,\,\,\,\, & \,\,\,\,\,\, -10\,\,\,\,\,\, & \,\,\,\,\,\, 0\,\,\,\,\,\, & \,\,\,\,\,\, 1\,\,\,\,\,\, & \,\,\,\,\,\, 0\,\,\,\,\,\, & \,\,\,\,\,\, 1\,\,\,\,\,\, & \,\,\,\,\,\, 0\,\,\,\,\,\, & \,\,\,\,\,\, 0\,\,\,\,\,\, & \,\,\,\,\,\, -1& \\
\sq{s}_4& 0\,\,\,\,\,\, & \,\,\,\,\,\, 11\,\,\,\,\,\, & \,\,\,\,\,\, -10\,\,\,\,\,\, & \,\,\,\,\,\, 0\,\,\,\,\,\, & \,\,\,\,\,\, 0\,\,\,\,\,\, & \,\,\,\,\,\, 0\,\,\,\,\,\, & \,\,\,\,\,\, 1\,\,\,\,\,\, & \,\,\,\,\,\, 0\,\,\,\,\,\, & \,\,\,\,\,\, 0\,\,\,\,\,\, & \,\,\,\,\,\, 0& \\
\sq{s}_1& 0\,\,\,\,\,\, & \,\,\,\,\,\, 0\,\,\,\,\,\, & \,\,\,\,\,\, 1\,\,\,\,\,\, & \,\,\,\,\,\, 0\,\,\,\,\,\, & \,\,\,\,\,\, 0\,\,\,\,\,\, & \,\,\,\,\,\, 0\,\,\,\,\,\, & \,\,\,\,\,\, 0\,\,\,\,\,\, & \,\,\,\,\,\, 1\,\,\,\,\,\, & \,\,\,\,\,\, -1\,\,\,\,\,\, & \,\,\,\,\,\, 0& \\
\sq{s}_2& 0\,\,\,\,\,\, & \,\,\,\,\,\, 0\,\,\,\,\,\, & \,\,\,\,\,\, 9\,\,\,\,\,\, & \,\,\,\,\,\, -9\,\,\,\,\,\, & \,\,\,\,\,\, 1\,\,\,\,\,\, & \,\,\,\,\,\, 0\,\,\,\,\,\, & \,\,\,\,\,\, 0\,\,\,\,\,\, & \,\,\,\,\,\, 1\,\,\,\,\,\, & \,\,\,\,\,\, 0\,\,\,\,\,\, & \,\,\,\,\,\, -1& \\
\sq{s}_3& 0\,\,\,\,\,\, & \,\,\,\,\,\, 0\,\,\,\,\,\, & \,\,\,\,\,\, 10\,\,\,\,\,\, & \,\,\,\,\,\, -9\,\,\,\,\,\, & \,\,\,\,\,\, 0\,\,\,\,\,\, & \,\,\,\,\,\, 0\,\,\,\,\,\, & \,\,\,\,\,\, 0\,\,\,\,\,\, & \,\,\,\,\,\, 1\,\,\,\,\,\, & \,\,\,\,\,\, 0\,\,\,\,\,\, & \,\,\,\,\,\, 0& \\
\sq{s}_1& 0\,\,\,\,\,\, & \,\,\,\,\,\, 0\,\,\,\,\,\, & \,\,\,\,\,\, 0\,\,\,\,\,\, & \,\,\,\,\,\, 8\,\,\,\,\,\, & \,\,\,\,\,\, -7\,\,\,\,\,\, & \,\,\,\,\,\, 0\,\,\,\,\,\, & \,\,\,\,\,\, 0\,\,\,\,\,\, & \,\,\,\,\,\, 0\,\,\,\,\,\, & \,\,\,\,\,\, 1\,\,\,\,\,\, & \,\,\,\,\,\, -1& \\
\sq{s}_2& 0\,\,\,\,\,\, & \,\,\,\,\,\, 0\,\,\,\,\,\, & \,\,\,\,\,\, 0\,\,\,\,\,\, & \,\,\,\,\,\, 9\,\,\,\,\,\, & \,\,\,\,\,\, -8\,\,\,\,\,\, & \,\,\,\,\,\, 0\,\,\,\,\,\, & \,\,\,\,\,\, 0\,\,\,\,\,\, & \,\,\,\,\,\, 0\,\,\,\,\,\, & \,\,\,\,\,\, 1\,\,\,\,\,\, & \,\,\,\,\,\, 0& \\
\sq{s}_1& 0\,\,\,\,\,\, & \,\,\,\,\,\, 0\,\,\,\,\,\, & \,\,\,\,\,\, 0\,\,\,\,\,\, & \,\,\,\,\,\, 0\,\,\,\,\,\, & \,\,\,\,\,\, 1\,\,\,\,\,\, & \,\,\,\,\,\, 0\,\,\,\,\,\, & \,\,\,\,\,\, 0\,\,\,\,\,\, & \,\,\,\,\,\, 0\,\,\,\,\,\, & \,\,\,\,\,\, 0\,\,\,\,\,\, & \,\,\,\,\,\, 1& \\
 \end{tabular}}}
 \vspace{.4cm}
\caption{Realizing rays of the~$2$-associahedron~$\multiassociahedron_{2,5}$, associated to each letter of the word~$\sq{c}^2\sq{w}_\circ(\sq{c})$ for~$n=5$. These rays were obtained by working on the coordinates of these in Table~\ref{tab:2associahedron5Perturbed}. The perturbation terms, that is the coordinates that are different from these obtained by fattening twice a suffix triangle in an initial triangle with coefficients~$\coefLeft=14-i-j$ and~$\coefRight=13-i-j$, appear boxed and red. In the non perturbed set of rays, all the corresponding terms are equal to zero. Observe finally that all perturbation terms are negative integers.}
\vspace{.3cm}
\label{tab:2associahedron5IntegerCoordinates}
\end{table}

Observing the pattern formed on the sorting network of the word~$\sq{c}^2\sq{w}_5(\sq{c})$ by these perturbation terms, we derived the conjectural pattern for realizing rays for any~$2$-associahedron of Figure~\ref{fig:2associahedronConjecturedPattern}, and described in Section~\ref{sec:introduction}. There we stated Question~\ref{que:fanRealizations} as a question rather than as a conjecture, because of the~$20$ bad ridges appearing with our choice of coefficients for~$n=8$. This is~\apriori not a problem since the number of perturbations grows quadratically with~$n$ and our candidate pattern of rays is still realizing for~$n=8$. Moreover the other tries we made seemed to indicate that more random integer perturbations failed realizing~$2$-associahedra before~$n=6$. Yet these bad ridges may grow quickly with~$n$ and let our pattern finally fail being realizing for all~$n$. But we still obtain realizations for some~$2$-associahedra with integer coordinates between~$-(2n+1)$ and~$(2n+2)$, for~$n\le8$.

\begin{theorem}
The rays of the pattern in Figure~\ref{fig:2associahedronConjecturedPattern} support a complete simplicial fan in~$\Rmbb^{2n}$ which realizes the multiassociahedron~$\multiassociahedron_{2,n}$ for~$n\in[8]$.
\end{theorem}

\begin{figure}
\centerline{\includegraphics[width=1.4\textwidth]{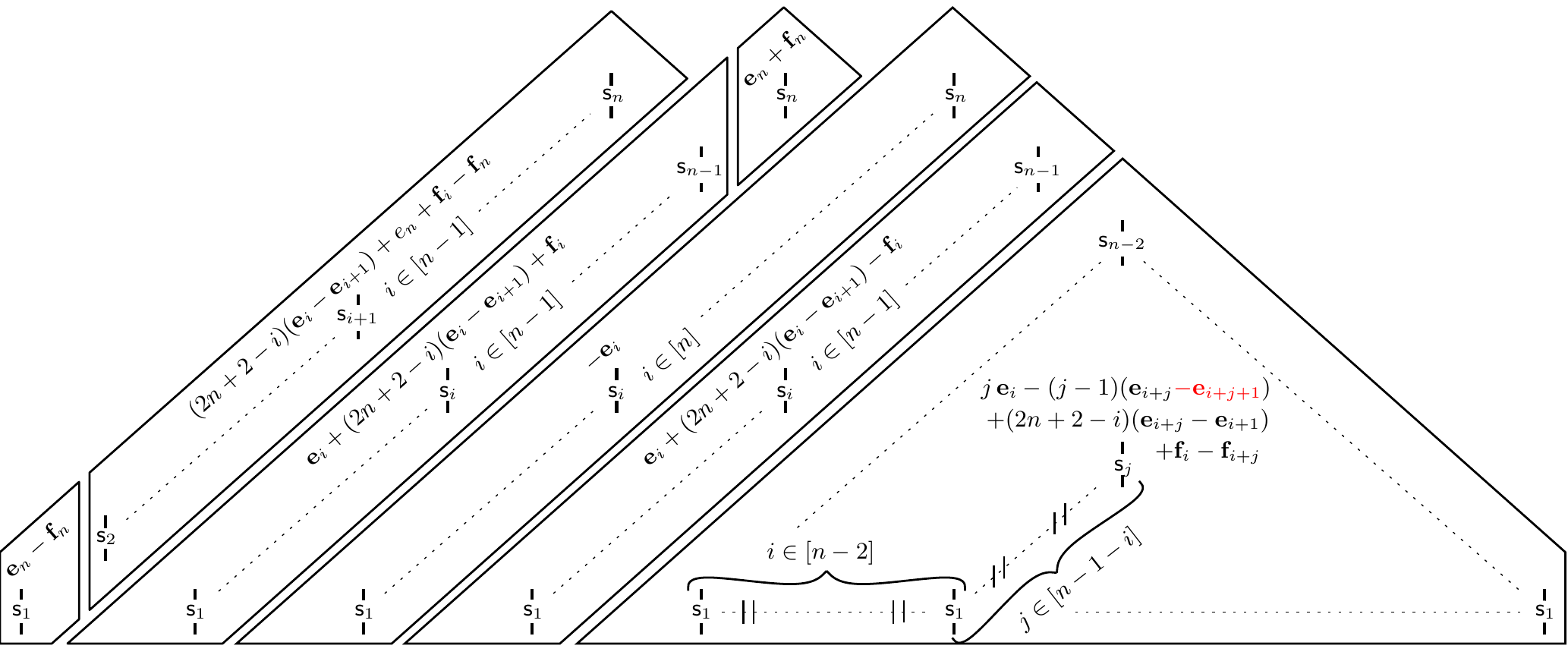}}
\caption{The candidate pattern for integer rays supporting a fan realizing~$2$-associahedra. We denote the~$n$ first vectors of the canonical basis of~$\Rmbb^{2n}$ by~$(\e_i)_{i\in[n]}$ and the~$n$ last ones by~$(\f_i)_{i\in[n]}$. The perturbation terms appear boxed and red. They are negative and replace zero coordinates of the non perturbed construction. This pattern is the one we obtain after applying~$2n$ times the rotation map to the underlying word~$\sq{c}^2\sq{w}_\circ(\sq{c})$ to have a better presentation.}
\label{fig:2associahedronConjecturedPattern}
\end{figure}

\section{Discussion}
\label{sec:discussion}

\subsection{Polytopality}

Unfortunately none of the new fans we produce happens to be the normal fan of a polytope. Not even in the case of the~$2$-associahedron~$\multiassociahedron_{2,3}$ which is known to have a polytopal realization by J.~Bokowski and V.~Pilaud~\cite{BokowskiPilaud}. Since all the transformations are intuitively chosen to fit with geometric constraints, it suggests that the way from combinatorics to geometry is very fine and confirms again how hard it is to handle with in the case of multiassociahedra.

\subsection{Further~$k$'s}

The few tries we made towards more general~$k$ than~$2$ did not work successfully and quickly produced sets of cones with many bad ridges, in contrast with Section~\ref{subsec:heuristicConstruction}. We tried to fatten three times a triangle, which produced really bad objects, and then we tried to fatten a triangle starting from the valid rays we had obtained after perturbing, which was not better. But our experiments lack exhaustive tries and the main issue with our method somehow comes from the fact that we have too many and too wide degrees of freedom to apply it.

%
%
\section*{Acknowledgments}

I am grateful to my advisor Vincent Pilaud for introducing me to subword complexes and multiassociahedra. I also want to warmly thank Lionel Pournin, Francisco Santos and again Vincent Pilaud for useful and helpful discussions and suggestions about the implementation of my geometric intuitions and the organization of this paper. In particular Lionel Pournin and Francisco Santos convinced me to keep on with this project as I was facing a dead end. I also thank Jean-Philippe Labb{\'e} for useful comments on the preprint version of this paper. Finally I acknowledge an anonymous referee for his or her relevant suggestions, corrections and comments.


\bibliographystyle{alpha}
\bibliography{bibliography}

\end{document}